\documentclass[reqno, 11pt]{amsart}

\numberwithin{equation}{section}

\usepackage{amssymb}
\usepackage{mathrsfs}
\usepackage{enumerate, xspace}
\usepackage{changebar}
\usepackage{hyperref}
\usepackage{pdfsync}
\usepackage{bm}
\usepackage{amsmath}

\newtheorem{thm}{Theorem}[section]
\newtheorem{lem}[thm]{Lemma}
\newtheorem{cor}[thm]{Corollary}

\newtheorem{prop}[thm]{Proposition}

\newtheorem{rem}[thm]{Remark}

\newcommand\cA{{\mathcal A}}

\newcommand\cC{{\mathcal C}}
\newcommand\cD{{\mathcal D}}
\newcommand\cE{{\mathcal E}}
\newcommand\cF{{\mathcal F}}
\newcommand\cG{{\mathcal G}}
\newcommand\cH{{\mathcal H}}

\newcommand\cL{{\mathcal L}}

\newcommand\cS{{\mathcal S}}

\newcommand\bR{{\mathbb R}}

\newcommand\ve{{\varepsilon}}

\newcommand\bdot{{\mathop{\odot}}}
\newcommand\tdot{{\mathop{\tilde\odot}}}

\begin{document}

\title[Isothermal limits in a temperature gradient]
{Non-equilibrium Isothermal transformations in a temperature
  gradient from a microscopic dynamics}

\author{Viviana Letizia}
 \address{Viviana Letizia\\
CEREMADE, UMR CNRS 7534\\
Universit\'e Paris-Dauphine, PSL\\
75775 Paris-Cedex 16, France}
 \email{{\tt letizia@ceremade.dauphine.fr}}

\author{Stefano Olla}
\address{Stefano Olla\\
CEREMADE, UMR CNRS 7534\\
Universit\'e Paris-Dauphine, PSL\\
75775 Paris-Cedex 16, France}
\email{{\tt olla@ceremade.dauphine.fr}}

\date{\today  {\bf File: {\jobname}.tex.}
} 

\begin{abstract}
We consider a chain of anharmonic oscillators immersed in a heat bath
with a temperature gradient and a time-varying tension applied to
one end of the chain while the other side is fixed to a point. 
We prove that under diffusive space-time rescaling the volume strain
distribution of the chain evolves following a non-linear diffusive
equation. The stationary states of the dynamics are of
non-equilibrium and have a positive entropy production, so the
classical relative entropy methods cannot be used. We develop new
estimates based on entropic hypocoercivity, that allow to control the
distribution of the position configurations of the chain.
The macroscopic limit can be used to model isothermal thermodynamic
transformations between non-equilibrium stationary states.  
\end{abstract}
\thanks{This work has been partially supported by the
  European Advanced Grant {\em Macroscopic Laws and Dynamical Systems}
  (MALADY) (ERC AdG 246953), and by the CAPES
    and CNPq program \emph{Science Without Borders}.}
\maketitle

\section{Introduction}
\label{sec:introduction}

Macroscopic isothermal thermodynamic transformations can be modeled
microscopically by putting a system in contact with Langevin heat bath
at a given temperature $\beta^{-1}$. In \cite{olla2014micro2} a chain
of $n$ anharmonic oscillators is \emph{immersed} in a heat bath of
Langevin thermostats acting independently on each
particle. Macroscopically equivalent isothermal dynamics is obtained
by elastic collisions with an external gas of independent particles
with maxwellian random velocities with variance $\beta^{-1}$.    
The effect is to quickly renew the velocities distribution of the
particles, so that at any given time it is very close to a maxwellian at
given temperature. The chain is pinned only on one side, while at the
opposite site a force (tension) $\tau$ is acting. The equilibrium
distribution is characterized by the two control parameters
$\beta^{-1}, \tau$ (temperature and tension). The total length and the
energy of the system in equilibrium are in general non-linear functions
of these parameters given by the standard thermodynamic relations. 

By changing the tension $\tau$ applied to the system, a new
equilibrium state, with the same temperature $\beta^{-1}$, will be eventually
reached. For large $n$, while the heat bath
equilibrates the velocities at the corresponding temperature at time of order 1, 
the system converges to this global equilibrium length at a time scale of order $n^2t$.
% much slower that the time scale the heat bath
% equilibrates the velocities at the corresponding temperature.  
In \cite{olla2014micro2} it is proven that the length stretch of the
system evolves in a diffusive space-time scale, i.e. after a scaling
limit the empirical distribution of the interparticle distances
converges to the solution of a non-linear diffusive equation governed
by the local tension. Consequently this diffusive equation describes
the non-reversible isothermal thermodynamic transformation from one
equilibrium to another with a different tension. By a further
rescaling of the time dependence of the changing tension,
a so called \emph{quasi-static} or \emph{reversible} isothermal
transformation  is obtained. Corresponding Clausius equalities/inequalities
relating work done and change in free energy can be proven.   

The results of \cite{olla2014micro2} summarized above concern isothermal
transformations from an \emph{equilibrium} state to another, by
changing the applied tension. In this article we are interested in transformations
between \emph{non-equilibrium} stationary states. We now consider the
chain of oscillators immersed in a heat bath with a \emph{macroscopic
  gradient} of temperature: each particle is in contact with
thermostats at a different temperature. These temperatures slowly change
from a particle to the neighboring one. A tension $\tau$ is again
applied to the chain. In the stationary state, that is now
characterized by the tension $\tau$ and the profile of temperatures
$\beta_1^{-1}, \dots, \beta_{n}^{-1}$, there is a continuous flow of
energy through the chain from the hot thermostats to the cold
ones. Unlike the equilibrium case, the probability distribution of 
the configurations of the chain in the stationary state cannot be
computed explicitly. 

By changing the applied tension we can obtain transitions from a
non-equilibrium stationary state to another, that will happen in a
diffusive space-time scale as in the equilibrium case. The main result
in the present article is that these transformations are again governed by a
diffusive equation that takes into account the local temperature
profile. The free energy can be computed according to the local
equilibrium rule and its changes during the transformation satisfy the
Clausius inequality with respect to the work done. This provides a
mathematically precise example for understanding non-equilibrium
thermodynamics from microscopic dynamics. 

The results in \cite{olla2014micro2} where obtained by using
the relative entropy method, first developed by H.T.Yau in
\cite{yau1991relative} for the Ginzburg-Landau dynamics, which is just
the over-damped version of the bulk dynamics of the oscillators
chain. The relative 
entropy method is very powerful and flexible, and was already applied
to interacting Ornstein-Uhlenbeck particles in the PhD thesis of
Tremoulet \cite{Tremoulet:2002p20753} as well as many other cases, in
particular in the hyperbolic scaling limit for Euler equation in the
smooth regime \cite{Olla:1993p69,BraxmeierEven:2014p16453}. 
This method consists in
looking at the time evolution of the relative
entropy of the distribution of the particle with respect to the local
Gibbs measure parametrized by the non-constant tension profile
corresponding to the solution of the macroscopic diffusion equation.
  The point of the method is in proving that the time derivative of
  such relative entropy is small, so that the relative entropy itself
  remains small with respect to the size of the system and local
  equilibrium, in a weak but sufficient form, propagates in time. 
In the particular applications to interacting  Ornstein-Uhlenbeck particles
\cite{Tremoulet:2002p20753,olla2014micro2}, the local Gibbs measure
needs to be corrected by a small recentering of the damped velocities due
to the local gradient of the tension.  
  
The relative entropy method seems to fail when the stationary measures
are not the equilibrium Gibbs measure, like in the present case. The
reason is that when taking the time derivative of the relative entropy
mentioned above, a large term, proportional to the gradient of the
temperature, appears. This term is related to the \emph{entropy 
production} of the stationary measure. Consequently we could not apply
the relative entropy method to the present problem.

A previous method was developed by Guo, Papanicolaou and Varadhan in
\cite{GPV} for over-damped dynamics. In this approach the main step
in closing the macroscopic equation is the direct comparison of the
coarse grained empirical density in the microscopic and macroscopic
space scale. They obtain first a bound of the Dirichlet form (more
precisely called Fisher information) from the
time derivative of the relative entropy with respect to the
equilibrium stationary measures. This bound implies that the system is
close to equilibrium on a local microscopic scale, and that the
density on a large microscopic interval is close to the density in a
small macroscopic interval (the so called \emph{one} and \emph{two block estimates},
see \cite{kipnis1999scaling} chapter 5). 

In the over-damped dynamics considered in \cite{GPV}, the Dirichlet
form appearing in the time derivative of the relative entropy controls
the gradients of the probability distributions with respects to the
position of the particles. In the damped models, 
the Dirichlet form appearing in the time derivative of the relative
entropy controls only the gradients on the velocities of the probability distribution
 of the particles. 
In order to deal with damped models a different approach for
comparing densities on the different scales was
developed in \cite{olla1991scaling}, after the 
over-damped case in \cite{varadhan1991scaling}, based on Young
measures. Unfortunately this approach requires a control of higher
moments of the density that are difficult to prove for lattice models.  
Consequently we could not apply this method either in the present
situation. 

The main mathematical novelty in the present article is the use of
entropic hypocoercivity, inspired by \cite{Villani:hypo}. 
We introduce a Fisher information form $I_n$ associated to the vector
fields $\{\partial_{p_i} + \partial_{q_i}\}_{i=1,\dots,n}$, defined by
\eqref{eq:hfi}. By computing the time derivative of this Fisher
information form on the distribution at time $t$ of the
configurations, we obtain a uniform bound $I_n\le Cn^{-1}$. This
implies that, at the macroscopic diffusive time scale, velocity gradients
of the distribution are very close to positions gradients. 
This allows to obtain a bound on the Fisher
information on the positions from the bound on the Fisher information
on the velocities. At this point we are essentially with the same
information as in the over-damped model, and we proceed as in \cite{GPV}.     
Observe that the Fisher information $I_n$ we introduce in \eqref{eq:hfi} is more
specific and a bit different than the distorted Fisher information
used by Villani in  \cite{Villani:hypo}, in particular $I_n$ is more
degenerate. On the other hand the calculations, that are contained in
appendix D are less \emph{miraculous} than in \cite{Villani:hypo}, and
they are stable enough to control the effect of the boundary tension
and of the gradient of temperature. This also suggests that entropic
hypocoercivity seems to be the right tool in order to obtain explicit
estimates uniform in the dimension of the system.
  
Adiabatic thermodynamic transformations are certainly more difficult
to be obtained from microscopic dynamics, for some preliminary results
see
\cite{Olla:1993p69,BraxmeierEven:2014p16453,bernardin2011transport,OllSim2015}.
 Equilibrium fluctuations for the dynamics with constant temperature can be treated
as in \cite{Olla:2003p91}. The fluctuations in the case with a gradient of
temperature are non-equilibrium fluctuation, and we believe that can be
treated with the techniques of the present article together with those
developed in the over-damped case in \cite{yau-chang}.

Large deviations for the stationary measure also require some further
mathematical investigations, but we conjecture that the corresponding
quasi-potential functional (\cite{bertini2012thermodynamic}) is given
by the  free energy associated to the local Gibbs measure, without any
non-local terms, unlike the case of the simple exclusion process. 

The article is structured in the following way. In \autoref{sec:model}  we define
the dynamics and we state the main result (Theorem
\autoref{theo1}). In \autoref{sec:non-equil-therm} we discuss
the consequences for the thermodynamic transformations from a
stationary state to another, the Clausius inequality and the
quasi-static limit. In \autoref{sec:entropyb} are obtained the bounds
on the entropy and the various Fisher informations needed in the proof
of the hydrodynamic limit. In \autoref{sec:limitpoint} we show
that any limit point of the distribution of the empirical density on
strain of the volume is concentrated in the weak solutions of the
macroscopic diffusion equation. The compactness, regularity and
uniqueness of the corresponding weak solution, necessary to conclude
the proof, are proven in the first three appendices. Appendix D
contains the calculations and estimates for the time derivative of the
Fisher information $I_n$.

\section{The dynamics and the results}
\label{sec:model}

We consider a chain of $n$ coupled oscillators in one dimension. Each
particle has the same mass, equal to one. The configuration in the
phase space is described by $\{q_i,p_i, i = 1, \dots,n\}\in \bR^{2n}$.
The interaction between two particles $i$ and $i-1$ is described by
the potential energy $V(q_i-q_{i-1})$ of an anharmonic spring. The
chain is attached on the left to a fixed point, so we set $q_0 =0,
p_0=0$.  
 We call $\{r_i=q_i-q_{i-1}, i=1,\dots, n\}$ the interparticle
 distance. 

We assume V to be a positive smooth function,
that satisfy the following assumptions:
\begin{enumerate}[i)]
\item 
\begin{equation}\label{eq:V}
\lim_{|r|\to \infty} \frac{V(r)}{|r|}=\infty, 
\end{equation}
\item there exists a constant $C_2>0$ such that:
  \begin{equation}
    \label{eq:bv2}
    \sup_{r} |V''(r)|\leq C_2,
  \end{equation}
\item there exists a constant $C_1>0$ such that:
  \begin{equation}
    \label{eq:45}
    V'(r)^2 \le C_1\left(1+ V(r)\right)).
  \end{equation}
\end{enumerate}
In particular these conditions imply $|V'(r)| \le C_0 + C_2 |r|$ for some constant $C_0$. 
Notice that this conditions allows potentials that may grow 
like $V(r) \sim |r|^\alpha$ for large $r$, with $1< \alpha \le 2$.

Energy is defined by the following Hamiltonian function:
\begin{equation}
\cH % (\mathbf{q},\mathbf{p})
:=\sum_{i=1}^n\left(\frac{p_i^2}{2}+V(r_i)\right) 
\end{equation}

The particle dynamics is subject to an interaction with an environment
given by Langevin heat bath at different temperatures $\beta_i^{-1}$. 
We choose $\beta_i$ as slowly varying on a macroscopic scale,
i.e. $\beta_i = \beta(i/n)$ for a given smooth strictly positive function
$\beta(x)$, $x\in [0,1]$ such that $\inf_{y\in [0,1]} \beta(y) \ge \beta_->0$.

The equations of motion are
given by
\begin{equation}\label{eq:eds1}
\begin{cases}
dr_i(t)=n^2(p_i(t)-p_{i-1}(t))dt\\
dp_i(t)=n^2(V'(r_{i+1}(t))-V'(r_i(t)))\; dt-n^2 \gamma p_i(t) dt +
n\sqrt{\frac{2\gamma}{\beta_i}}dw_i(t), \quad i=1,..,N-1\\
dp_n(t)=n^2(\bar{\tau}(t)-V'(r_n(t)))\;dt-n^2\gamma p_n(t)\; dt +
n\sqrt{\frac{2\gamma}{\beta_n}} dw_n(t) .
\end{cases}
\end{equation}
Here $\{w_i(t)\}_i$ are $n$-independent Wiener processes, $\gamma>0$ is
the coupling parameter with the Langevin thermostats. 
The time is rescaled according to the diffusive space-time scaling,
i.e. $t$ is the macroscopic time.
The tension $\tilde{\tau}=\tilde{\tau}(t)$ changes at the
macroscopic time scale (i.e. very slowly in the microscopic time
scale). The generator of the diffusion is given by 
\begin{equation}\label{eq:gen}
\cL^{\bar{\tau}(t)}_n:= n^2 \cA^{\bar{\tau}(t)}_n+n^2\gamma \cS_n,
\end{equation}
where $\cA^{\bar{\tau}}_n$ is the Liouville generator 
\begin{equation}
A_n^{\bar{\tau}}=\sum_{i=1}^n (p_i-p_{i-1})\partial_{r_i} +\sum_{i=1}^{n-1}(V'(r_{i+1})-V'(r_i))\partial_{p_i}+(\bar{\tau}-V'(r_n))\partial_{p_n}
\end{equation}
while $S_n$ is the operator 
\begin{equation}
S_n=\sum_{i=1}^n\left(\beta^{-1}_i\partial_{p_i}^2-p_i\partial_{p_i}\right)
\end{equation}

% An equivalent dynamics is given by a different modeling of the heat
% bath: particle $i$ undergoes to stochastic elastic collisions with
% \emph{particles} 
% of the environment at temperature $\beta^{-1}_i$, i.e. a independent
% exponentially distributed times of intensity $\gamma n^2$ particle $i$
% change its velocity to a new random velocity normally distributed with
% variance $\beta^{-1}_i$. The evolution equations are given by: 
% \begin{equation}\label{eq:intcoll}
% \begin{cases}
% dr_i(t)=n^2(p_i(t)-p_{i-1}(t))dt\\
% dp_i(t)=n^2(V'(r_{i+1}(t))-V'(r_i(t)))dt + \left(\tilde
%   p_{i,N_i(n^2 t)} - p(t^-)\right) dN_i(\gamma n^2t) \\ 
% %\quad i=1,..,n-1\\
% dp_n(t)=n^2(\bar{\tau}(t)-V'(r_n))dt + \left(\tilde
%   p_{n,N_n(n^2 t)} - p_n(t^-)\right) dN_n(\gamma n^2 t)
% \end{cases}
% \end{equation}
% where $\tilde p_{i,k}$ are independent gaussian variables on mean zero
% and variance $\beta^{-1}_i$, and $\{N_i(t), i=1,\dots,n\}$ are
% independent Poisson processes of intensity 1.
% All result presented in this article can be extended in a
% straightforward way to the dynamics \eqref{eq:intcoll}.

\subsection{Gibbs measures}
For $\bar{\tau}(t)=\tau$ constant, and $\beta_i=\beta$ homogeneous, the system has a unique invariant
probability measure given by a product of invariant Gibbs
measures $\mu_{\tau,\beta}^n$:
\begin{equation}
d\mu_{\tau,\beta}^n =\prod_{i=1}^n 
e^{-\beta(\mathcal{E}_i-\tau r_i)-\mathcal{G}(\tau,\beta)} \; dr_idp_i 
% =g_\tau^n d\mathbf{q} d\mathbf{p} 
\end{equation}
where $\cE_i$ is the energy of the particle $i$:
\begin{equation}
\mathcal{E}_i=\frac{p_i^2}{2}+V(r_i).
\end{equation}
%and $d\mathbf{q} d\mathbf{p}=\prod_{i=1}^n dr_idp_i$.
The function $\mathcal{G}(\tau,\beta)$ is the Gibbs potential defined as:
\begin{equation}
\mathcal{G}(\tau,\beta) =\log{\left[\sqrt{2\pi\beta^{-1}}\int
    e^{-\beta(V(r)-\tau r)} dr \right]}.
\end{equation}
Notice that, thanks to condition \eqref{eq:V}, $\mathcal{G}(\tau,\beta)$ is finite for any $\tau\in \mathbb R$ 
and any $\beta>0$. Furthermore it is strictly convex in $\tau$.

The free energy of the equilibrium state $(r,\beta)$ is given by the
Legendre transform of $\beta^{-1}\cG(\tau,\beta)$:
\begin{equation}
\cF(r,\beta) = \sup_{\tau}\{ \tau r -\beta^{-1}\cG(\tau,\beta)\}
\end{equation}
The corresponding convex conjugate variables are the equilibrium average length
\begin{equation}\label{grandpot}
\mathfrak{r}(\tau,\beta)=\beta^{-1}\partial_\tau \cG(\tau,\beta)
\end{equation} 
and the tension 
\begin{equation}
\bm{\tau}(r,\beta)=\partial_r \cF(r,\beta).
\end{equation} 
Observe that 
\begin{equation}
\mathbb{E}_{\mu_{\tau,\beta}^n}[r_i]=\mathfrak{r}(\tau, \beta),
\qquad
\mathbb{E}_{\mu_{\tau,\beta}^n}[V'(r_i)]= \tau.
\end{equation}
\\

% We are interested in the macroscopic behavior of the elongation, at
% time $t$, as $n \to \infty$, so we are going to study the hydrodynamic
% limit of $r$, notice that $t$ is already the
% macroscopic time, since we have already multiplied the generator by
% $n^2$.

\subsection{The hydrodynamic limit}

% Given initially the profiles of temperature
% $\beta^{-1}(x)$ and tension $\tau_0(x)$ for $x \in [0,1]$, 
 We assume that for a given initial profile $r_0(x)$
% the system in an initial configuration 
% given by the local Gibbs measures (non
% homogeneous product):
% \begin{equation}
% d\mu^n_{\tau(0,\cdot),\beta(0)}=\prod_{i=1}^n
% e^{-\beta(i/n)(\cE_i-\tau(i/n)r_i)-\cG_{\tau(0,i/n),\beta(i/n)}}dr_idp_i 
% \end{equation}
% This implies the following convergence in probability with respect to
the initial probability distribution satisfies:
\begin{equation}\label{eq:inprof}
\frac{1}{n}\sum_{i=1}^n G(i/n) r_i(0)\ \mathop{\longrightarrow}_{n\to\infty}\ 
 \int_0^1 G(x) r_0(x) dx  \qquad \text{in probability} 
\end{equation}
for any continuous test function $G\in \cC_0([0,1])$.
We expect that this same convergence happens at the macroscopic time $t$:
\begin{equation}
\frac{1}{n}\sum_{i=1}^n G(i/n) r_i(t)\longrightarrow \int_0^1 G(x) r(x,t)dx
\end{equation}
where $r(x,t)$ satisfies the following
diffusive equation
\begin{equation}\label{eqn:hydro}
\begin{cases}
&\partial_t r(x,t)=\frac{1}{\gamma}\partial_x^2 \bm{\tau}(r(x,t),
\beta(x)) \qquad 
\mbox{for} \qquad x\in[0,1]\\
& \partial_x \bm\tau(r(t,x),\beta(x))|_{x=0} = 0, \quad \bm\tau (r(t,x),
\beta(x)) |_{x=1} =\bar{\tau}(t),
\quad t>0\\
&r(0,x) = r_0(x), \quad x\in[0,1]
\end{cases}
\end{equation}

We say that $r(x,t)$ is a weak solution of \eqref{eqn:hydro} if for
any smooth function $G(x)$ on $[0,1]$ such that $G(1)=0$ and $G'(0)=0$
we have
\begin{equation}
  \label{eq:weak}
  \int_0^1 G(x) \left(r(x,t) - r_0(x)\right) dx = \gamma^{-1}\int_0^t
  ds  \left[\int_0^1 G''(x) \mathbf{\tau}(r(x,s), \beta(x)) dx -
    G'(1) \bar\tau(s) \right].
\end{equation}
In appendix C we prove that the weak solution is unique in the class
of functions such that:  
  \begin{equation}
    \label{eq:2a}
    \int_0^t ds \int_0^1 \left(\partial_x\tau(r(x,s),
      \beta(x))\right)^2 dx < +\infty .
  \end{equation}

Let $\nu^n_{\beta_\cdot}$ the inhomogeneous Gibbs measure
\begin{equation}
d\nu^n_{\beta_\cdot}= \prod_{i=1}^n \frac{e^{-\beta_i\cE_i}}{{Z_{\beta_i}}}
\end{equation}
Observe that this is \textbf{not} the stationary measure for the
dynamics defined by \eqref{eq:eds1} and~\eqref{eq:gen} for $\bar\tau = 0$.

Let $f^n_t$ the density, with respect to $\nu^n_{\beta_\cdot}$, of the
probability distribution of the system at time t, i.e. the solution of
\begin{equation}
  \label{eq:fk}
  \partial_t f^n_t = \mathcal L_n^{\bar\tau(t),*} f^n_t ,
\end{equation}
where $\mathcal L_n^{\bar\tau(t),*}$ is the adjoint of 
$\mathcal L_n^{\bar\tau(t)}$ with respect to $\nu^n_{\beta_\cdot}$,
i.e. explicitly 
\begin{equation}
\cL_n^{\bar{\tau}(t),*}= - n^2 \cA_n^{\tau(t)} - n\sum_{i=1}^{n-1} \nabla_n
\beta(i/n) p_iV'(r_{i+1}) + n^2\beta(1)p_n\bar{\tau} + n^2\gamma
\cS_n,
\label{eq:adjointop}
\end{equation}
where 
\begin{equation}
\nabla_n \beta(i/n) =  n\left(\beta\left(\frac{i+1}{n}\right) - \beta\left(\frac in\right)\right), \qquad i=1, \dots, n-1 .
\label{eq:39}
\end{equation}

Define the relative entropy of $f^n_td\nu^n_{\beta_\cdot}$ with respect
to $d\nu^n_{\beta_\cdot}$ as:
\begin{equation}
H_n (t) = \int f^n_t\log{f_t^n} d\nu^n_{\beta_\cdot}.
\end{equation}
We assume that the initial density $f^n_0$ satisfy the bound
\begin{equation}
  \label{eq:31}
  H_n (0) \le Cn.
\end{equation}
We also need some regularity of  $f^n_0$: define the hypercoercive
Fisher information functional:
\begin{equation}
  \label{eq:hfi}
  \begin{split}
    I_n(t) &= \sum_{i=1}^{n-1} \beta^{-1}_i \int
    \frac{\left( (\partial_{p_i} 
        + \partial_{q_i} ) f^n_t\right)^2}{f^n_t} 
   d\nu_{\beta\cdot} % \\
    % &=
    % \sum_{i=1}^{n-1} \left( \cD_{i,n}^p(f) + \cD_{i,n}^r(f) +
    %   2\beta_i^{-1} \int \frac{(\partial_{q_i} f) (\partial_{p_i}
    %     f)}{f} \right) d\nu_{\beta\cdot} \ge 0
  \end{split}
\end{equation}
where $\partial_{q_i} = \partial_{r_{i}} - \partial_{r_{i+1}},
i=1,\dots,n-1$, and $\nu_{\beta\cdot}:=\nu^n_{\beta\cdot}$. We assume that
\begin{equation}
  \label{eq:32}
  I_n(0) \le K_n 
\end{equation}
with $K_n$ growing less than exponentially in $n$. We will show in
Appendix D that for any $t>0$ we have $I_n(t) \le Cn^{-1}$. 

Furthermore we assume that 
\begin{equation}\label{eq:incond}
\lim_{n\to\infty} \int \left|\frac{1}{n}\sum_{i=1}^n
    G\left(\frac{i}{n}\right)r_i -\int_0^1G(x){r_0}(x) dx
  \right| f^n_0 d\nu_{\beta\cdot} = 0
\end{equation}
for any continuous test function $G\in \cC_0([0,1])$.

\begin{thm}\label{theo1}
Assume that the starting initial distribution satisfy the above
conditions. 
Then
\begin{equation}
\lim_{n\to \infty}\int \left|\frac{1}{n}\sum_{i=1}^n
    G\left(\frac{i}{n}\right)r_i -\int_0^1G(x){r(x,t)} dx
  \right| f^n_t d\nu_{\beta\cdot} = 0, 
\end{equation}
where $r(x,t)$ is the unique weak solution of (\ref{eqn:hydro})
satisfying \eqref{eq:2a}.
\end{thm}

Furthermore a local equilibrium result is valid in the following sense:
consider a local function $\phi(\bf r,\bf p)$ such that for some positive
 finite constants $C_1, C_2$ we have the bound 

 \begin{equation}
|\phi(\mathbf r,\mathbf p) |\le C_1 \sum_{i\in\Lambda_\phi} (p_i^2 +
V(r_i))^\alpha +C_2, \qquad \alpha < 1 \label{eq:enb}
\end{equation}
where $\Lambda_\phi$ is the local support of $\phi$.
Let $k_\phi$ the length of $\Lambda_\phi$, and
let $\theta_i \phi$ be the shifted function, well defined for $k_\phi < i< n-k_\phi$, and define
  \begin{equation}
    \label{eq:13}
    \hat\phi(r,\beta) = \mathbb E_{\mu_{ \bm\tau(r,\beta) ,\beta}}\left(\phi\right).
 %   \qquad \tau= \bm\tau(r,\beta) 
  \end{equation}

\begin{cor}
\begin{equation}\label{eq:locequi}
\lim_{n\to \infty} \int \left|\frac1n \sum_{i=k_\phi+1}^{n-k_\phi}
    G\left(\frac{i}{n}\right)\theta_i\phi({\bf r},{\bf p}) -\int_0^1G(x)
    \hat\phi(r(x,t),\beta(x)) dx
  \right|f^n_t d\nu_{\beta\cdot} = 0, 
\end{equation}
\end{cor}

\section{Non-equilibrium thermodynamics}
\label{sec:non-equil-therm}

We collect in this section some interesting consequences of the main
theorem for the non-equilibrium thermodynamics of this system. All
statements contained in this section can be proven rigorously, except for
one that will require more investigation in the future. The aim is to
build a non equilibrium thermodynamics in the spirit of
\cite{bertini2013clausius,bertini2012thermodynamic}. The equilibrium
version of these results has been already proven in
\cite{olla2014micro2}. 

As we already mentioned, stationary states of our dynamics are not
given by Gibbs measures if a gradient in the temperature profile is
present, but they are still characterized by the tension $\bar\tau$
applied. We denote these stationary distributions as \emph{non-equilibrium stationary states} (NESS).
 Let us denote $f^n_{ss,\tau}$ the density of the stationary
distribution with respect to $\nu_{\beta_\cdot}$.

It is easy to see that
\begin{equation}
  \label{eq:15}
  \int V'(r_i) f^n_{ss,\tau}\nu_{\beta_\cdot} = \tau, \qquad i=1,\dots,n.
\end{equation}
In fact, since $\int p_i f^n_{ss,\tau}\nu_{\beta_\cdot} = 0$ and
\begin{equation*}
  \begin{split}
    n^{-2} \mathcal L_n^\tau p_i &= V'(r_{i+1}) - V'(r_i) - \gamma p_i,
    \quad i=1,\dots,n-1,\\
    n^{-2} \mathcal L_n^\tau p_n &= \tau - V'(r_n) - \gamma p_n,
  \end{split}
\end{equation*}
we have
\begin{equation*}
   0 = \int (V'(r_{i+1}) - V'(r_i))  f^n_{ss,\tau}\nu_{\beta_\cdot} =  \int (\tau - V'(r_n))  f^n_{ss,\tau}\nu_{\beta_\cdot}.
\end{equation*}
By the main theorem \ref{theo1}, there exists a stationary profile of stretch $r_{ss,\tau}(y) =
\mathfrak r(\tau, \beta(y))$ (defined by \eqref{grandpot}) 
such that  for any continuous test function $G$:
\begin{equation}
  \label{eq:38}
  \lim_{n\to \infty} \int \left|\frac{1}{n}\sum_{i=1}^n
    G\left(\frac{i}{n}\right)r_i -\int_0^1G(x)
    r_{ss,\tau}(x) dx \right|f^n_{ss,\tau} d\nu_{\beta_\cdot} = 0,
\end{equation}

In order to study the transition from one stationary state to another
with different tension, we start the system at time $0$ with a
stationary state with tension 
$\tau_0$, and we change tension with time, setting $\bar\tau(t) =
\tau_1$ for $t\ge t_1$. The distribution of the system will eventually
converge  to a stationary state with tension $\tau_1$. 
Let $r(x,t)$ be the solution of the
macroscopic equation \eqref{eq:weak} starting with $r_0(x) =  r_{ss,\tau_0}(x)$. 
Clearly $r(x,t) \to r_1(x) = r_{ss,\tau_1}(x)$, as $t\to\infty$.

\subsection{Excess Heat}
\label{sec:excess-heat}

The (normalized) total internal energy of the system is defined
by 
\begin{equation}
  \label{eq:14}
  U_n := \frac 1n \sum_{i=1}^n\left(\frac{p_i^2}{2}+V(r_i)\right) 
\end{equation}
It evolves as:
\begin{equation*}
   U_n(t) -  U_n(0) = \mathcal W_n(t) + Q_n(t)
\end{equation*}
where 
\begin{equation*}
  \mathcal W_n(t) = \int_0^t \bar\tau(s) np_n(s) ds = \int_0^t
  \bar\tau(s) \frac{dq_n(s)}n %=  \int_0^t
  % \bar\tau(s) d\mathcal L_n(s) = \mathcal W(t)
\end{equation*}
is the (normalized) work done by the force $\bar\tau(s)$ up to time
$t$, % and 
% $\mathcal L_n(t) = {q_n(t)}/n$,
while
\begin{equation}\label{eq:heat}
  Q_n(t) = \gamma\; n\sum_{j=1}^n \int_0^t ds \left(p_j^2(s) -
    \beta_j^{-1}\right) + \sum_{j=1}^n \sqrt{2\gamma
    \beta^{-1}_j} \int_0^t p_j(s) dw_i(s).
\end{equation}
is the total flux of energy between the system and the heat bath
(divided by $n$). 
As a consequence of theorem \autoref{theo1} we have that 
\begin{equation*}
  \lim_{n\to\infty} \mathcal W_n(t) = \int_0^t \bar\tau(s) d\mathcal
  L(s)
\end{equation*}
  where $\mathcal L(t) = \int_0^1 r(x,t) dx$, the total macroscopic
  length at time $t$. While for the energy difference we expect that
  \begin{equation}\label{eq:unproven}
    \begin{split}
      \lim_{n\to\infty} \left(U_n(t) - U_n(0)\right) &= \int_0^1  
      \left[u({\bm \tau}(r(x,t),\beta(x)),\beta(x)) - 
        u(\tau_0, \beta(x))\right] dx% \\ 
     % & = \int_0^1 u(\bm\tau(r(x,t)),\beta(x)) dx - u_{ss, \tau_0}
    \end{split}
  \end{equation}
where $u(\tau,\beta)$ is the average energy for
$\mu_{\beta,\tau}$, i.e.
\begin{equation*}
  u(\tau,\beta) = \int \mathcal E_1 d\mu^1_{\tau, \beta} = 
  \frac 1{2\beta} + \int V(r) e^{-\beta (V(r) - \tau r) -
    \mathcal{\tilde G}(\tau,\beta)} dr 
\end{equation*}
 with $\mathcal{\tilde G}(\tau,\beta) = \log \int e^{-\beta (V(r) -
 \tau r)} dr$. 
Unfortunately \eqref{eq:unproven} does not follow from \eqref{eq:locequi}, since 
\eqref{eq:enb} is not satisfied.
Consequently at the moment we do not have a rigorous proof of \eqref{eq:unproven}. 
In the constant temperature profile case,
treated in \cite{olla2014micro2}, this limit can be computed
rigorously thanks to the use on the relative entropy method
\cite{yau1991relative} that gives a better control on the local
distribution of the energy.

% Notice that it is easy to prove that 
% \begin{equation*}
%   \lim_{n\to\infty} \frac 1n \sum_{i=1}^n \frac{p_i^2}{2} = \int_0^1
%   \beta^{-1}(x) dx
% \end{equation*}

Since ${\bm \tau}(r(x,t),\beta(x)) \to \tau_1$ as $t\to\infty$, % the macroscopic equation will reach
% the stationary state at tension $\tau_1$, so that
it follows that
$$
u({\bm \tau}(r(x,t),\beta(x)),\beta(x)) \to u(\tau_1,\beta(x))
$$
and the energy change will become
\begin{equation}
  \label{eq:1stprinciple}
 \int_0^1 \left(u(\tau_1,\beta(x))  - u(\tau_0,\beta(x))\right) dx =
 \int_0^{+\infty} \bar\tau(s) d\mathcal  L(s) ds  + Q = \mathcal W + Q
\end{equation}
where $Q$ is the limit of \eqref{eq:heat}, which is called \emph{excess
  heat}. So equation \eqref{eq:1stprinciple} is the expression of the
first principle of thermodynamics in this \emph{isothermal}
transformation between non--equilibrium stationary states. Here
\emph{isothermal} means that the profile of temperature does not
change in time during the transformation. 

\subsection{Free energy}
\label{sec:free-energy}

Define the \emph{free energy} associated to the macroscopic profile $r(x,t)$:
\begin{equation}
  \label{eq:16}
  \widetilde{\mathcal F}(t) = \int_0^1 \mathcal F(r(x,t), \beta(x)) dx.
\end{equation}
Correspondingly the free energy associated to the macroscopic stationary state is:
\begin{equation}
  \label{eq:18}
   \widetilde{\mathcal F}_{ss}(\tau) = \int_0^1 \mathcal F(r_{ss,\tau}(x),
   \beta(x)) dx 
\end{equation}
A straightforward calculation using \eqref{eq:weak} gives
\begin{equation}
  \label{eq:17}
  \begin{split}
    \widetilde{\mathcal F}(t) -  \widetilde{\mathcal F}_{ss}(\tau_0) = 
    \mathcal W(t) - \gamma^{-1} \int_0^t ds \int_0^1 \left( \partial_x
      \bm\tau(r(x,s),\beta(x))\right)^2 dx  
  \end{split}
\end{equation}
and after the time limit $t\to\infty$
\begin{equation}
  \label{eq:2ndprinciple}
  \begin{split}
    \widetilde{\mathcal F}_{ss}(\tau_1) -  \widetilde{\mathcal
      F}_{ss}(\tau_0) =  
    \mathcal W - \gamma^{-1} \int_0^{+\infty} dt \int_0^1 \left( \partial_x
      \bm\tau(r(x,t),\beta(x))\right)^2 dx  \\
    \le \mathcal W
  \end{split}
\end{equation}
  i.e. Clausius inequality for NESS. Notice that in the case $\beta_j$ constant, 
this is just the usual Clausius inequality (see \cite{olla2014micro2}).

\subsection{Quasi-static limit and reversible transformations}
\label{sec:quasi-static-limit-1}

The thermodynamic transformation obtained above from the stationary state at tension $\tau_0$ 
to the one at tension $\tau_1$ is an irreversible transformation, where the work done on 
the system by the external force is strictly bigger than the change in free energy. 

In thermodynamics the \emph{quasi-static} transformations are (vaguely) defined as those processes 
where changes are so slow such that the system is in \emph{equilibrium} at each instant of time.
In the spirit of \cite{bertini2013clausius} and \cite{olla2014micro2},
these quasi static transformations are precisely defined as a limiting process
by rescaling the time dependence of
the driving tension $\bar\tau$ by a small parameter $\varepsilon$,
i.e. by choosing $\bar\tau(\epsilon t)$. Of course the right time
scale at which the evolution appears is $\varepsilon^{-1}t$ and the
rescaled solution $\tilde r^\varepsilon(x,t) = r(x,\varepsilon^{-1}t)$
satisfy the equation
\begin{equation}
  \label{eq:edpquasistatic}
  \begin{cases}
&\partial_t \tilde r^\varepsilon(x,t)=\frac{1}{\epsilon\gamma}\partial_x^2
\bm{\tau}(\tilde r^\varepsilon(x,t), 
\beta(x)) \qquad 
\mbox{for} \qquad x\in[0,1]\\
& \partial_x \bm\tau(\tilde r^\varepsilon(t,x),\beta(x))|_{x=0} = 0, \quad
\bm\tau (\tilde r^\varepsilon(t,x), 
\beta(x)) |_{x=1} =\bar{\tau}(t),
\quad t>0\\
&\mathbf{\tau}(\tilde r^\varepsilon(0,x),\beta(x)) = \tau_0, \quad x\in[0,1]
\end{cases}
\end{equation}
By repeating the argument above, equation \eqref{eq:2ndprinciple} became:
\begin{equation}
  \label{eq:40}
   \begin{split}
    \widetilde{\mathcal F}_{ss}(\tau_1) -  \widetilde{\mathcal
      F}_{ss}(\tau_0) =  
    \mathcal W^\varepsilon - \frac 1{\epsilon\gamma} \int_0^{+\infty}
    dt \int_0^1 \left( \partial_x 
      \bm\tau(\tilde r^\varepsilon(x,t),\beta(x))\right)^2 dx
\end{split}
\end{equation}
By the same argument used in \cite{olla2014micro2} for $\beta$
constant, it can be proven that the last term on the right hand side of
\eqref{eq:40} converges to $0$ as $\varepsilon\to 0$, and that 
$\bm\tau(\tilde r^\varepsilon(x,t),\beta(x)) \to \bar\tau(t)$ for
almost any $x\in [0,1]$ and $t\ge 0$. 
Consequently in the quasi-static limit we have the Clausius equality
$$ 
\widetilde{\mathcal F}_{ss}(\tau_1)
-  \widetilde{\mathcal F}_{ss}(\tau_0) =  \mathcal W
$$
This implies the following equality for the heat in the quasi-static limit:
\begin{equation}
  \label{eq:19}
  Q = \int_0^1 \beta^{-1}(x)
  \left(S(r_{ss}(x,\tau_1),u_{ss}(x,\tau_1)) -
    S(r_{ss}(x,\tau_0),u_{ss}(x,\tau_0))\right) dx
\end{equation}
analogous of the equilibrium equality $Q=T\Delta S$.

In \cite{demasi} a direct quasi-static limit is obtained form the microscopic dynamics without passing through the
macroscopic equation \eqref{eq:weak}, by choosing a driving tension $\bar\tau$ that changes at a slower time scale.

\section{Entropy and hypercoercive bounds}
\label{sec:entropyb}

In this section we prove the bounds on the relative entropy and the different Fisher informations 
that we need in the proof of the hydrodynamic limit in section \autoref{sec:limitpoint}. These bounds provide 
a quantitative information on the closeness of the local distributions of the particles
 to some equilibrium measure.

In order to shorten formulas, we introduce here some vectorial
notation. Given two vectors $u=(u_1, \dots, u_n), v=(v_1, \dots, v_n)$, define
\begin{equation*}
 u \bdot v = \sum_{i=1}^{n} \beta^{-1}_{i} u_i v_i,\qquad  u
 \tilde\bdot v = \sum_{i=1}^{n-1} \beta^{-1}_{i} u_i v_i, \qquad
 |u|_\bdot^2 =  u \bdot u, \quad 
  |u|_{\tilde\bdot}^2 =  u \tilde\bdot u.  
\end{equation*}
We also use the notations
\begin{equation}
  \label{eq:23}
  \begin{split}
    \partial_p = (\partial_{p_1}, \dots, \partial_{p_n}) \qquad 
     \partial^*_p = (\partial^*_{p_1}, \dots,
     \partial^*_{p_n}), 
     \qquad \partial^*_{p_i} = \beta_i p_i - \partial_{p_i}\\
    \partial_q = (\partial_{q_1}, \dots, \partial_{q_n}),
    \qquad \partial_{q_i} = \partial_{r_i} - \partial_{r_{i+1}}, 
    \quad \partial_{q_n} = \partial_{r_n}.
       \end{split}
\end{equation} 
Observe that with this notations we can write
\begin{equation}
  \label{eq:25}
\mathcal S_n = - \partial^*_p \odot \partial_p, \qquad  
  \mathcal A_n^\tau = p\cdot \partial_q - \partial_q \mathcal V
  \cdot \partial_p  + \tau \partial_{p_n}
\end{equation}
where $\mathcal V = \sum_i V(r_i)$ and the $\cdot$ denotes the usual
scalar product in $\mathbb R^n$. 
Then we define the following Fisher informations forms on a 
probability density distribution (with respect to $\nu_{\beta\cdot}$):
\begin{equation}
  \label{eq:fisher}
  \begin{split}
    \cD^p_n(f) & = \int \frac{|\partial_p f|^2_{\bdot}}{f}  d\nu_{\beta\cdot}, \qquad
    \tilde{ \cD}^p_n(f)  = \int \frac{|\partial_p f|^2_{\tilde\bdot}}{f}  d\nu_{\beta\cdot}\\
    \cD^r_n(f) & = \int \frac{|\partial_q f|^2_{\tilde\bdot}}{f}  d\nu_{\beta\cdot}\\
    I_n(f) &= \int \frac{|\partial_p f+\partial_q f|^2_{\tilde\bdot}}{f} d\nu_{\beta\cdot} 
    = \tilde{\cD}^p_n(f) + \cD^r_n(f) +  2\int \frac{\partial_{q} f \tilde\bdot \partial_{p} f}{f} 
      d\nu_{\beta\cdot} \ge 0
  \end{split}
\end{equation}

\begin{prop}\label{hypoentropy}
 Let $f_t^n$ the solution of the forward equation \eqref{eq:fk}.
 Then there exist a constant $C$ such that
  \begin{equation}
    \label{eq:20}
   H_n(t) \le Cn, \qquad \int_0^t  \cD^p_n (f_s^n) ds \le \frac{C}{n}, \qquad
   \int_0^t  \cD^r_n (f_s^n) ds \le \frac{C}{n}.
  \end{equation}
\end{prop}

\begin{proof}

Taking the time derivative of the entropy we obtain:
\begin{equation}
\frac d{dt} H_n(t)= 
\int (\cL_n^{\bar{\tau}(t)})^* f_t^n \log{f_t^n}d\nu_{\beta_\cdot}
\end{equation}

So that, using \eqref{eq:adjointop}, we have
\begin{equation}\label{eq:dtH}
\begin{split}
\frac d{dt} H_n(t)&=\int f_t^n \cL_n^{\bar{\tau}(t)}\log{f_t^n}d\nu_{\beta_\cdot} 
=\int n^2 \cA_n^{\bar\tau(t)} f d\nu_{\beta_\cdot} - \gamma n^2
 \cD^p_n(f_t^n) \\
&=- n\sum_{i=1}^{n-1}\nabla_n \beta(i/n) \int V'(r_{i+1}) p_i  f_t^n
d\nu_{\beta_\cdot} + n^2 \beta_n \bar{\tau}(t) \int p_n  f_t^n
d\nu_{\beta_\cdot} -\gamma n^2 \cD^p_{n}(f_t^n)
\end{split}
\end{equation}
% where we have denoted
% \begin{equation*}
%   \cD^p_{i,n}(f) = \beta_i^{-1}\int \frac{(\partial_{p_i} f)^2}{f}
%   d\nu_{\beta\cdot}  
% \end{equation*}
% The second term on the right hand side of \eqref{eq:dtH} is
% proportional, after time integration, to the work done by the tension,
% so it can be easily estimate in the following way.
Recall that 
$q_n = \sum_{i=1}^n r_i$,  then the time integral of the second term on the RHS of \eqref{eq:dtH} gives
\begin{equation}
  \label{eq:37}
  \begin{split}
    n^2 \beta_n \int_0^t ds \; \bar{\tau}(s)\int p_n f_s^n
    d\nu_{\beta_\cdot} = \beta_n \int_0^t ds\; \bar{\tau}(s)\int
    \mathcal L_n^{\bar\tau(s)} q_n f_s^n d\nu_{\beta_\cdot} \\
    = \beta_n \bar{\tau}(t)\int q_n f_t^n d\nu_{\beta_\cdot} - 
    \beta_n \bar{\tau}(0)\int q_n f_0^n d\nu_{\beta_\cdot} -
    \beta_n \int_0^t ds\; \bar{\tau}'(s)\int
    q_n f_s^n d\nu_{\beta_\cdot}
  \end{split}
\end{equation}
By the entropy inequality, for any $a_1>0$, using the first of the conditions \eqref{eq:V},
\begin{equation}
  \label{eq:4}
  \begin{split}
    \int |q_n| f_s^n d\nu_{\beta_\cdot} \le \frac 1{a_1} \log \int
    e^{a_1|q_n|} d\nu_{\beta_\cdot} + \frac 1{a_1} H_n(s) 
    \le  \frac 1{a_1} \log \int
    \prod_{i=1}^n e^{a_1|r_i|} d\nu_{\beta_\cdot} + \frac 1{a_1} H_n(s) \\
    \le   \frac 1{a_1} \sum_{i=1}^n \log \int \left(e^{a_1 r_i} + e^{-a_1 r_i} \right) d\nu_{\beta_\cdot} + \frac 1{a_1} H_n(s) \\
    =  \frac 1{a_1} \sum_{i=1}^n \left(\mathcal G(a_1,\beta_i) + \mathcal G(-a_1,\beta_i) - 2 \mathcal G(0,\beta_i) \right) 
    + \frac 1{a_1} H_n(s) 
    \le n C(a_1, \beta_{\cdot} ) +  \frac 1{a_1} H_n(s) 
  \end{split}
\end{equation}
We 
apply \eqref{eq:4} to the three terms of the RHS of \eqref{eq:37}. So after this time
integration we can estimate, for any $a_1>0$,
\begin{equation}
  \label{eq:5}
  \begin{split}
    n^2 \beta(1) \left|\int_0^t ds\ \bar{\tau}(t) \int p_n f_t^n
      d\nu_{\beta_\cdot} \right| \le \frac {\beta(1)
      K_{\bar\tau}}{a_1} \left(H_n(t) + H_n(0) + \int_0^t H_n(s) ds  \right) \\
+ n (2+t) \beta(1) K_{\bar\tau} C(a_1, \beta_{\cdot} ) 
  \end{split}
\end{equation}
where $K_{\bar\tau} = \sup_{s>0} \left(|\bar\tau(s)| + |\bar\tau'(s)|\right)$.

By integration by part and Schwarz inequality, for any $a_2>0$ we have
\begin{equation*}
  \begin{split}
    \left|n\sum_{i=1}^{n-1}\nabla_n \beta(i/n) \int V'(r_{i+1}) p_i
      f_t^n d\nu_{\beta_\cdot}\right|  = 
 \left|n\sum_{i=1}^{n-1} \frac{\nabla_n \beta(i/n)}{\beta(i/n)} \int V'(r_{i+1}) \partial_{p_i}
      f_t^n d\nu_{\beta_\cdot}\right|\\
\le \frac 1{2a_2}
    \sum_{i=1}^{n-1}\frac{(\nabla_n \beta(i/n))^2}{\beta_i} \int
    V'(r_{i+1})^2 f_t^n d\nu_{\beta_\cdot} 
    + \frac {a_2 n^2}{2} \tilde{\cD}^p_{n}(f_t^n)
  \end{split}
\end{equation*}

By our assumptions on $\beta(\cdot)$ and assumption \eqref{eq:45} on $V$, 
we have that for some constant $C_{\beta_\cdot} >0$
 depending on $\beta(\cdot)$ and $V$,
\begin{equation}\label{eq:energyb1}
  \sum_{i=1}^{n-1}\frac{(\nabla_n \beta(i/n))^2}{\beta_i} 
    V'(r_{i+1})^2 \le C_{\beta_\cdot}\sum_{i=1}^{n-1}
    V'(r_{i+1})^2 \le C_{\beta_\cdot} C_1 \sum_{i=1}^{n} \left(V(r_i) +1\right) 
\end{equation}
By the entropy inequality, for any $\delta$ such that  $0< \delta < \inf_y \beta(y)$, there exists a finite constant 
$C_{\delta,\beta_\cdot}$ depending on $V,\delta$ and $\beta(\cdot)$ such that:
\begin{equation}\label{eq:energyb2}
  \begin{split}
    \sum_{i=1}^{n}\int V(r_i) f_t^n d\nu_{\beta_\cdot} &\le \frac
    1\delta \log \int e^{\delta \sum_{i=1}^{n}\int V(r_i)}
    d\nu_{\beta_\cdot} + \frac 1\delta H_n(t)\\
    & = \frac
    1\delta \sum_{i=1}^{n} \left(\mathcal G(0,\beta_i - \delta) - \mathcal G(0,\beta_i)\right) + \frac 1\delta H_n(t)
  \le C_{\delta,\beta_\cdot} n + \frac 1\delta H_n(t) 
  \end{split}
\end{equation}

At this point we have obtained the following inequality, for some constant $C$ not depending on $n$,
\begin{equation}\label{eq:ebp}
  \begin{split}
    H_n(t) - H_n(0) \le - n^2 \left(\gamma - \frac {a_2}2\right)
   \int_0^t  \cD^p_{n}(f_s^n) ds + \left(\frac{C_{\beta_\cdot}}{2 a_2 \delta} +
     \frac{\beta(1) K_{\bar\tau}}{a_1}\right) \int_0^t H_n(s) ds \\
   +  \frac{\beta(1) K_{\bar\tau}}{a_1} 
   \left(H_n(t) + H_n(0)\right) + n c(a_1,a_2,\delta, \bar \tau, \beta_\cdot)
  \end{split}
\end{equation}
consequently, choosing $a_2=\gamma$ and $a_1= 2 \beta(1) K_{\bar\tau}$, we have
\begin{equation}\label{eq:eb2}
  H_n(t) \le 3 H_n(0) + C' \int_0^t H_n(s) ds + cn - n^2 \gamma \int_0^t
   \cD^p_{n}(f_s^n) ds 
\end{equation}
where $C'$ and $c$ are constants independent of $n$. Given the initial
bound on $H_n(0) \le cn$, by Gronwall inequality we have for some $c''$ independent on $n$:
\begin{equation}
H_n(t)
\le c''e^{C't} n\label{eq:entropyb}.
\end{equation}
Inserting this in \eqref{eq:eb2} we obtain, for some $\tilde C$ independent of n,
\begin{equation}
  \label{eq:1}
  {\gamma}\int_0^t \cD^p_{n}(f_s^n) ds \le
  \frac{\tilde C}{n} 
\end{equation}

The bound \eqref{eq:1} gives only informations about the distribution
of the velocities, but we actually need a corresponding bound of the
distribution of the positions. 

% Define for $i=1,\dots,n-1$
% \begin{equation}
%   \label{eq:qfisher}
%   \cD_{i,n}^r(f) = \beta_i^{-1} \int \frac{(\partial_{q_i} f)^2}{f}
%   d\nu_{\beta\cdot}, 
% \end{equation}
% where $\partial_{q_i} = (\partial_{r_i} - \partial_{r_{i+1}})$. % and
% % $\partial_{q_n} = \partial_{r_n}$.

% Recall the definition of $I_n(t)$ given at \eqref{eq:hfi}, and that we
% assumed $I_n(0)$ growing less than exponential.

In appendix D we prove that, as a consequence of \eqref{eq:1}, we have
\begin{equation}
  \label{eq:33}
  I_n(t) \le \frac Cn \qquad \forall t>0.
\end{equation}

Consequently
\begin{equation*}
  \begin{split}
     \cD_{n}^r(f_t^n) = I_n(f_t^n) - \tilde{\cD}_{n}^p(f_t^n) -   
     2\int \frac{\partial_{q} f_t^n \tilde\bdot \partial_{p} f_t^n}{f_t^n} 
      d\nu_{\beta\cdot} \\
      \le  \frac Cn - \tilde{\cD}_{n}^p(f_t^n)  -  2\int \frac{\partial_{q} f_t^n \tilde\bdot \partial_{p} f_t^n}{f_t^n} 
      d\nu_{\beta\cdot} \\
    \le \frac Cn -  \tilde{\cD}_{n}^p(f_t^n) + 2 \tilde{\cD}_{n}^p(f_t^n) + \frac 12 \cD_{n}^r(f_t^n) 
  \end{split}
\end{equation*}
that gives
\begin{equation*}
  \cD_{n}^r(f_t^n) \le 2   \tilde{\cD}_{n}^p(f_t^n) + \frac {2C}n
\end{equation*}
Since we have already the bound \eqref{eq:1}, \eqref{eq:20} follows. \end{proof}

% This bound allow us to apply directly the one/two
% blocks method developed in \cite{GPV} for the over-damped dynamics.

\section{Characterization of the limit points}
\label{sec:limitpoint}

Define the empirical measure 
$$
\pi_t^n(dx):=\frac{1}{n}\sum_{i=1}^n r_i(t)\delta_{i/n}(dx) .
$$
and we use the notation, for a given smooth function $G:[0,1]\to \mathbb{R}$,
$$
\langle \pi_t^n,G \rangle := \frac{1}{n}\sum_{i=1}^n G\left(\frac{i}{n}\right) r_i(t)
$$

Computing the time derivative we have:
\begin{equation}
\langle \pi_t^n,G \rangle- \langle \pi_0^n, G\rangle
= \int_0^t \frac{1}{n}\sum_{i=1}^n G\left(\frac{i}{n}\right)\mathcal{L}_n^{\tilde{\tau}(t)}r_i(t)
\end{equation}
Since
\begin{equation*}
\mathcal{L}_n^{\tilde{\tau}(t)}r_i = n^2(p_i-p_{i-1}), \qquad i=1,\dots,n, \quad p_0=0,
\end{equation*}
after performing a summation by parts, we obtain 
\begin{equation}\label{empevo}
\mathcal{L}_n^{\bar{\tau}(t)}\langle \pi^n_t, G\rangle
= -\sum_{i=1}^{n-1} \nabla_n G \left(\frac{i}{n}\right) p_i(t) + n p_n(t) G(1) .
\end{equation}
  where $\nabla_n G$ is defined by \eqref{eq:39}. 
% $$
% \nabla_n G \left(\frac{i}{n}\right)=
% n\left[G\left(\frac{i+1}{n}\right)-G\left(\frac{i}{n}\right)\right] ,
% \qquad i=1, \dots, n-1.
% $$
We define also 
$$
\nabla_n^* G \left(\frac{i}{n}\right)=
n\left[G\left(\frac{i-1}{n}\right)-G\left(\frac{i}{n}\right)\right] 
\qquad i = 2, \dots, n. 
$$

% The first summation by parts canceled one of the factors $n$, a second
% summation by parts is also possible because of the empirical momentum
% evolution, that is

Now observe that
\begin{equation}
\begin{split}
 \mathcal{L}_n^{\bar{\tau}(t)}&\left[\frac{1}{n^2}\sum_{i=1}^{n-1} \nabla_n G
\left(\frac{i}{n}\right)p_i - \frac 1n p_n G(1)\right]  
= - \gamma\sum_{i=1}^{n-1} \nabla_n G\left(\frac{i}{n}\right) p_i  + \gamma n
p_n G(1) \\
&+ \sum_{i=1}^{n-1} \nabla_n G\left(\frac{i}{n}\right) 
(V'(r_{i+1})-V'(r_i)) - n G\left(1\right) (\bar\tau(t) - V'(r_n)) \\
=& - \gamma\sum_{i=1}^{n-1} \nabla_n G\left(\frac{i}{n}\right) p_i  + \gamma n
p_n G(1) \\
&+\frac 1n \sum_{i=2}^{n-1} \nabla_n^* \nabla_n G\left(\frac{i}{n}\right)
V'(r_{i+1}) + \nabla_n G\left(\frac{n-1}{n}\right) V'(r_n)  
- \nabla_n G\left(\frac{1}{n}\right) V'(r_1)\\
&- n G\left(1\right) (\bar\tau(t) - V'(r_n))
\end{split}
\end{equation}
Recall that, by the weak formulation of the macroscopic equation, cf. \eqref{eq:weak}, 
it is enough to consider test functions $G$ such that $G(1) = 0$ and
$G'(0) = 0$. This takes care of the last term on the RHS of the above expression and in \eqref{empevo}, 
and putting these two expression together and dividing by $\gamma$, we obtain
\begin{equation}\label{eq:exact}
  \begin{split}
    \mathcal{L}_n^{\bar{\tau}(t)}\langle \pi^n, G\rangle =
     \frac 1{\gamma n} \sum_{i=2}^{n-1} (-\nabla_n^* \nabla_n) 
    G\left(\frac{i}{n}\right) V'(r_{i+1}) - 
   \gamma^{-1} \nabla_n G\left(\frac{n-1}{n}\right) V'(r_n) \\
    +\gamma^{-1} \nabla_n G\left(\frac{1}{n}\right) V'(r_1) +  
    \mathcal{L}_n^{\bar{\tau}(t)}\frac{1}{\gamma n^2}\sum_{i=1}^{n-1} \nabla_n G
    \left(\frac{i}{n}\right)p_i
  \end{split}
\end{equation}
It is easy to show, by using the entropy inequality, that the last two terms are negligible.
In fact, since $G'(0)=0$ we have that $|\nabla_n G\left(\frac{1}{n}\right)|\le C_G n^{-1}$. Furthermore 
\begin{equation*}
  \int e^{\alpha |V'(r)| - \beta_1 V(r)} dr < +\infty \qquad \forall
  \alpha >0.
\end{equation*}
Then, using the entropy inequality we have for any $\alpha>0$:
\begin{equation}\label{eq:bb0}
  \begin{split}
    \int \left|\gamma^{-1} \nabla_n G\left(\frac{1}{n}\right) V'(r_1) \right| f^n_s d\nu_{\beta_\cdot}
    \le \frac{C_G}{n\gamma}\int |V'(r_1)| f^n_s d\nu_{\beta_\cdot} \\
    \le \frac{C_G}{n\gamma \alpha} \int e^{\alpha |V'(r_1)|} d\nu^n_{\beta_\cdot}
    + \frac{C_G}{n\gamma \alpha} H_n(s) \le \frac{C(\alpha)}n + \frac{C'}{\alpha}
  \end{split}
\end{equation}
that goes to $0$ after taking the limit as $n\to\infty$ then
$\alpha\to\infty$. 
About the last term of the RHS in \eqref{eq:exact}, after time
integration we have to estimate
\begin{equation*}
  \int \frac{1}{\gamma n^2}\sum_{i=1}^{n-1} \left|\nabla_n G
    \left(\frac{i}{n}\right)\right| |p_i| f^n_s d\nu_{\beta_\cdot} 
\end{equation*}
for $s=0,t$. By similar use of the entropy inequality it follows that also this term disappear when $n\to\infty$.

To deal with the second term of the RHS of \eqref{eq:exact}, we need
the following lemma:
\begin{lem}\label{bd}
  \begin{equation}
    \label{eq:7}
    \lim_{n\to\infty} \mathbb E\left( \left|\int_0^t \int \left( V'(r_n(s)) - \bar\tau(s)\right) ds\right| \right) = 0
  \end{equation}
\end{lem}
\begin{proof}%[Proof of Lemma \ref{bd}]

Observe that
\begin{equation}\label{eq:boundyn}
  V'(r_n) - \bar\tau(s) = - \frac 1{n^2} \mathcal L^{\bar\tau(s)} p_n - \gamma p_n = - \frac 1{n^2} \mathcal L^{\bar\tau(s)} (p_n + \gamma q_n). 
\end{equation}
Then after time integration:
\begin{equation*}
  \begin{split}
    \int_0^t  \left( V'(r_n(s)) - \bar\tau(s)\right) ds
    = \frac 1{n^2} \left( p_n(0)  -  p_n(t) \right) - 
    \frac{\gamma}{n^2}  (q_n(t) - q_n(0)) + \frac{\sqrt{2 \gamma \beta_n}}{n} w_n(t).
  \end{split}
\end{equation*}

% \begin{equation*}
%   \begin{split}
%     \int_0^t \int \left( V'(r_n) - \bar\tau(s)\right) ds
%     f_s^n\; d\nu_{\beta_\cdot}  = \frac 1{n^2} \left(\int p_n f_0^n\;
%       d\nu_{\beta_\cdot} - \int p_n f_t^n\; d\nu_{\beta_\cdot} \right) - 
%     \gamma \int_0^t ds \int  p_n f_s^n\; d\nu_{\beta_\cdot}  \\
%     = \frac 1{n^2} \left(\int p_n f_0^n\;
%       d\nu_{\beta_\cdot}  - \int p_n f_t^n\; d\nu_{\beta_\cdot} \right) - 
%     \gamma \int_0^t ds \int  \partial_{p_n}f_s^n\; d\nu_{\beta_\cdot} 
%   \end{split}
% \end{equation*}
It is easy to show that, using similar estimate as \eqref{eq:37} and \eqref{eq:4},
 the expectation of the absolute value of
 the right hand side of the above expression converges to $0$   as $n\to\infty$.
\end{proof}
It follows that
\begin{equation}
  \label{eq:41}
  \lim_{n\to\infty}  \mathbb E\left(\left|\int_0^t \left(\nabla_n G\left(\frac{n-1}{n}\right)  V'(r_n(s))   
- G'(1) \bar\tau(s)\right) ds\right|\right) = 0.
\end{equation}

% \section{Closing the macroscopic equation}
% \label{sec:clos-macr-equat}

We are finally left to deal with the first term of the RHS of
\eqref{eq:exact}. We will proceed as in \cite{GPV}. For any
$\varepsilon >0$ define
\begin{equation}
  \label{eq:34}
  \bar r_{i,\varepsilon} = \frac 1{2n \varepsilon +1}\sum_{|j-i|\le
    n\varepsilon} r_j, \qquad n\varepsilon < i < n (1-\varepsilon). 
\end{equation}
We first prove that the boundary terms are negligible:
\begin{lem}\label{lem-boundary}
  \begin{equation}
  \label{eq:35}
  \lim_{\varepsilon \to 0} \lim_{n\to\infty}\int_0^t \int \left|\frac
   1{\gamma n} \left(\sum_{i=2}^{[n\varepsilon]} +
     \sum_{i=[n(1-\varepsilon)] +1}^{[n-1]} \right) 
    (-\nabla_n^* \nabla_n) G\left(\frac{i}{n}\right) 
     V'(r_{i+1}) \right|\;
    f_s^n\; d\nu_{\beta_\cdot} \; ds = 0
\end{equation}
\end{lem}

\begin{proof}
  For simplicity of notation let us estimate just one side. Since our conditions on $V$ imply that 
  $|V'(r)| \le C_2 |r| + C_0$, we only need to prove that for any $t \ge 0$:
  \begin{equation}
    \label{eq:42}
    \lim_{\varepsilon \to 0} \lim_{n\to\infty} \int \frac 1{n} \sum_{i=2}^{[n\varepsilon]} |r_{i}| \;
    f_t^n\; d\nu_{\beta_\cdot} = 0
  \end{equation}
By the entropy inequality we have:
\begin{equation*}
  \begin{split}
    \int \frac 1{n} \sum_{i=2}^{[n\varepsilon]} |r_{i}| \; f_t^n\; d\nu_{\beta_\cdot} \le
    \frac{1}{n\alpha} \log \int \prod_{i=2}^{[n\varepsilon]} e^{\alpha |r_i|} d\nu_{\beta_\cdot} + \frac{H_n(t)}{\alpha n}\\
    \le \frac{1}{n\alpha} \sum_{i=2}^{[n\varepsilon]} 
\left(\mathcal G(\alpha,\beta_i) + \mathcal G(-\alpha,\beta_i) - 2 \mathcal G(0,\beta_i)\right) + \frac C\alpha
  \end{split}
\end{equation*}
Since $\mathcal G(\alpha,\beta_i) + \mathcal G(-\alpha,\beta_i) - 2 \mathcal G(0,\beta_i) \le C' \alpha^2$, 
for a constant $C'$ independent on $i$, we have
\begin{equation*}
  \int \frac 1{n} \sum_{i=2}^{[n\varepsilon]} |r_{i}| \; f_t^n\; d\nu_{\beta_\cdot} \le C' \varepsilon \alpha + \frac C\alpha,
\end{equation*}
 and by choosing $\alpha = \varepsilon^{-1/2}$ \eqref{eq:42} follows.
\end{proof}

% This last limit can be proven again by the entropy inequality.

We are only left to show that
\begin{equation}
  \label{eq:6}
  \begin{split}
   \lim_{\varepsilon \to 0} \lim_{n\to\infty} \int_0^t \int 
   \left|\frac
   1{\gamma n} \sum_{i=[n\varepsilon]+1}^{[n(1-\varepsilon)]} 
    (-\nabla_n^* \nabla_n) G\left(\frac{i}{n}\right) 
    \left( V'(r_{i+1}) - \bm \tau(\bar r_{i,\varepsilon}, \beta_i) \right) \right|\;
    f_s^n\; d\nu_\cdot \; ds = 0
  \end{split}
\end{equation}
Thanks to the bound \eqref{eq:20}, we are now in the same position as
in the proof of the over-damped dynamics, as considered in \cite{GPV},
and by using similar argument as used there (the so called
one-block/two blocks) \eqref{eq:6} follows. A slight difference is due
to the dependence of $\bm \tau$ on $\beta_i$, but since this changes
very slowly and smoothly in space it is easy to consider microscopic
blocks of size $k$ with constant temperature inside.

At this point the proof of
theorem \autoref{theo1} follows by standard arguments.
Let $Q_n$ the probability distribution of $\pi^n_\cdot$ on 
$\mathcal C([0,T], \mathcal M([0,1])$, 
where $\mathcal M([0,1])$ are the signed 
measures on $[0,1]$. In appendix B we prove that the sequence $Q_n$ is
compact. Then, by the above results any limit point $Q$ of
$Q_n$ is concentrated on absolutely continuous measures with densities
$\bar r(y,t)$ such that for any $0\le t\le T$, 
\begin{equation}
\begin{split}
  \label{eq:36}
  \mathbb E^Q\Big|\int_0^1 G(y) &\left(\bar r(y,t) -\bar r(y,0) \right)
      dy\\
      &- \gamma^{-1}\int_0^t ds \left[\int_0^1 G{''}(y)  
        \bm \tau(\bar r(y,s), \beta(y)) dy - G'(1) \bar\tau(s)\right]
  \Big| = 0
\end{split}
\end{equation}
Furthermore in appendix A we prove that $Q$ is concentrated on
densities that satisfy the regularity condition to have uniqueness of
the solution of the equation.

\section{Appendix A: Proof of the regularity bound \ref{eq:2a}}
\label{sec:appendix-b:-proof}

\begin{prop}
There exists a finite constant $C$ such that for any limit point
distribution $Q$ we have the bound:
  \begin{equation}
    \label{eq:9}
    \mathbb E^Q \left( \int_0^t ds \int_0^1 dx \left(\partial_x 
        \bm{\tau}(\bar r(s,x),\beta(x)) \right)^2 \right) < C .
  \end{equation}
\end{prop}

\begin{proof}
  It is enough to prove that for any function $F \in \mathcal
  C^1([0,1])$ such that $F(0) = 0$ the following inequality holds:
  \begin{equation}
    \label{eq:10}
    \mathbb E^Q \left( \int_0^t ds \left[\int_0^1 dx F'(x)
        \bm{\tau}(\bar r(s,x),\beta(x)) - F(1)\bar\tau(s)\right]
    \right) \le C \left(\int_0^1 F(x)^2 dx\right)^{1/2} .
  \end{equation}
In fact by a duality argument, since $\bm{\tau}(\bar r(s,1),\beta(1)) = \bar\tau(s)$,  we have:
\begin{equation*}
  \begin{split}
    \int_0^1 dx \left(\partial_x \bm{\tau}(\bar r(s,x),\beta(x)) \right)^2 =
    \sup_{F \in \mathcal C^1([0,1])} \frac{\int_0^1 dx F'(x) \bm{\tau}(\bar r(s,x),\beta(x)) - F(1)\bar\tau(s)}{\int_0^1 F(x)^2 dx}.
  \end{split}
\end{equation*}
Observe that \eqref{eq:10} corresponds to a choice of test functions $G(x)$ in  \eqref{eq:weak}
 such that $G' = F$. In order to obtain \eqref{eq:10}, compute
\begin{equation*}
  \begin{split}
    \frac 1{n^2} \mathcal L_n^{\bar\tau} \sum_{i=1}^n F(i/n) (p_i +
    \gamma q_i) 
    = \sum_{i=1}^n F(i/n) A_n^{\bar\tau} p_i \\
    = \sum_{i=1}^{n-1} F(i/n) \left(V'(r_{i+1}) - V'(r_i)\right) 
    + F(1) \left(\bar\tau - V'(r_n)\right) \\
    = \frac 1n \sum_{i=2}^{n} \nabla_n^* F(i/n) V'(r_i)  + F(1)
    \bar\tau - F(1/n) V'(r_1)
  \end{split}
\end{equation*}
and after time integration and averaging over trajectories we have
\begin{equation}\label{eq:reg}
  \begin{split}
    \frac{1}{n^2} \int \sum_{i=1}^n F(i/n) (p_i +
    \gamma q_i) (f^n_t - f^n_0) d\nu_{\beta_\cdot} \\
    = \int_0^t ds \int \frac 1n \sum_{i=2}^{n} \nabla_n^* F(i/n) V'(r_i) f^n_s
    d\nu_{\beta_\cdot} +  F(1)\int_0^t \bar\tau(s)\; ds \\
    - F(1/n) \int_0^t ds \int V'(r_1) f^n_s d\nu_{\beta_\cdot}.
  \end{split}
\end{equation}
It is easy to see that, since $F(0) = 0$ and differentiable, the last
term of the right hand side is negligible as $n\to\infty$, by the same argument used in \eqref{eq:bb0}.

About the first term on the RHS of \eqref{eq:reg}, 
by the results of \autoref{sec:limitpoint}, it converges,
through subsequences, to
\begin{equation*}
  -\mathbb E^Q \left( \int_0^t ds \int_0^1 dx F'(x)
        \bm{\tau}(\bar r(s,x),\beta(x))\right).
\end{equation*}
About the left hand side of \eqref{eq:reg}, one can see easily that
\begin{equation*}
   \frac{1}{n^2} \int \sum_{i=1}^n F(i/n) p_i (f_t^n - f_0^n)
   d\nu_{\beta_\cdot} \mathop{\longrightarrow}_{n\to\infty} 0 .
\end{equation*}
Using the inequality $\sum_i q_i^2 \le n^2 \sum_i r_i^2$,
we can bound the other term of the LHS of \eqref{eq:reg} by observing that, for $s=0,t$, 
\begin{equation*}
  \begin{split}
     \left|\frac{\gamma}{n} \int \sum_{i=1}^n F(i/n) \frac{q_i}{n} f_s^n d\nu_{\beta_\cdot} \right| \le
   \gamma\left(\frac 1n\sum_{i=1}^n F(i/n)^2\right)^{1/2} 
   \left(\int\frac 1n\sum_{i=1}^n \frac{q_i^2}{n^2} f_s^n d\nu_{\beta_\cdot} \right)^{1/2} \\
   \le \gamma\left(\frac 1n\sum_{i=1}^n F(i/n)^2\right)^{1/2} 
   \left(\int\frac 1n\sum_{i=1}^n r_i^2 f_s^n d\nu_{\beta_\cdot}\right)^{1/2} \le C\gamma \left(\frac
     1n\sum_{i=1}^n F(i/n)^2\right)^{1/2}.
  \end{split}
\end{equation*}
Since $F$ is a continuous function on $[0,1]$ the rhs of the above expression is bounded in $n$ 
and converges to the $L^2$ norm of $F$ as $n\to\infty$. Thus \eqref{eq:10} follows.
\end{proof}

\section{Appendix B: Compactness}
\label{sec:append-c:-comp}

We prove in this section that the sequence of probability distributions $Q_n$ 
on $\mathcal C([0,t], \mathcal M)$ induced by $\pi_n$ is tight. Here $\mathcal M$ is the space 
of the signed measures on $[0,1]$ endowed by the weak convergence topology.
This tightness is consequence of the following statement. 
\begin{prop}
  For any function $G\in \mathcal C^1([0,1])$ such that $G(1)=0$, $G'(0)=0$
and any $\ve>0$ we have
  \begin{equation}
    \label{eq:11}
    \lim_{\delta\to 0} \limsup_{n\to\infty} \mathbb P^{\mu_0} 
    \left[ \sup_{0\le s<t\le T, |s-t| <\delta} \left| 
        <\pi_n(t),G> -  <\pi_n(s),G> \right| \ge \ve \right] = 0
  \end{equation}
\end{prop}

\begin{proof}
By doing similar calculations as done in \autoref{sec:limitpoint} (see \eqref{empevo} and following ones)
  \begin{equation*}
    \begin{split}
      & <\pi_n(t),G> - <\pi_n(s),G> = -\int_s^t du \sum_{i=1}^{n-1}
      \nabla_n G \left(\frac{i}{n}\right) p_i(u) \\
     & = \int_s^t du \frac 1{\gamma n} \sum_{i=2}^{n-1}
      (-\nabla_n^*\nabla_n) G\left(\frac 1n\right) V'(r_{i+1}(u))
      - \int_s^t du \frac 1{\gamma}
      \nabla_n G\left(\frac {n-1}n\right) V'(r_n(u))\\
     &\qquad + \int_s^t du \frac 1{\gamma}
      \nabla_n G\left(\frac {1}n\right) V'(r_1(u))
      + \frac 1{\gamma n^2} \sum_{i=2}^{n-1}
      \nabla_n G\left(\frac in\right) (p_i(t) - p_i(s))\\
      &\qquad + \frac 1n \sum_{i}^n \sqrt{2\gamma\beta_j^{-1}} \nabla_n
      G\left(\frac in\right) \left(w_i(t) - w_i(s)\right)\\
      &\qquad := I_1(s,t) + I_2(s,t) + I_3(s,t) + I_4(s,t) + I_5(s,t)
    \end{split}
  \end{equation*}
We treat the corresponding 5 terms separately.
The term $I_3 =  \int_s^t du \frac 1{\gamma}
      \nabla_n G\left(\frac {1}n\right) V'(r_1(u))$ is the easiest to estimate, since $G'(0)= 0$, and using Schwarz inequality we have
\begin{equation*}
  \begin{split}
    \sup_{0\le s<t\le T, |s-t| <\delta}|I_3(s,t)| &\le \sup_{0\le
      s<t\le T, |s-t| <\delta} \frac{C}{n\gamma} \int_s^t |V'(r_1(u))| du \\
    &\le \sup_{0\le s<t\le T, |s-t| <\delta} \frac{C}{n\gamma} |t-s|^{1/2} \left(\int_s^t |V'(r_1(u))|^2 du\right)^{1/2} \\
    &\le \frac {C \delta^{1/2}}{n\gamma} \left(\int_0^T
      |V'(r_1(u))|^2 du\right)^{1/2}.
  \end{split}
\end{equation*}
Since, by entropy inequality,
\begin{equation*}
  \begin{split}
    \mathbb E \left[\left(\int_0^T |V'(r_1(u))|^2
        du\right)^{1/2}\right] \le \left[\int_0^T \mathbb
      E\left(|V'(r_1(u))|^2 \right) du\right]^{1/2} \\
    \le C \left[\int_0^T \mathbb E \left(\sum_{i=1}^n
        (V(r_i(u)) +1) \right) du\right]^{1/2} \le C T^{1/2} n^{1/2}
  \end{split}
\end{equation*}
so that 
\begin{equation*}
   \mathbb E \left[\sup_{0\le s<t\le T, |s-t| <\delta}|I_3(s,t)| \right] \le \frac
 {C\delta^{1/2} T^{1/2}}{\gamma n^{1/2}} \mathop{\longrightarrow}_{n\to\infty} 0.
\end{equation*}

About $I_2$, this is equal to 
\begin{equation}\label{eq:i2}
  - \frac 1{\gamma}
      \nabla_n G\left(\frac {n-1}n\right) \int_s^t du \left(V'(r_n(u))-
        \bar\tau(u) \right)
        - \frac 1{\gamma}
      \nabla_n G\left(\frac {n-1}n\right) \int_s^t du \bar\tau(u) 
\end{equation}
The second term of the above expression is trivially bounded by $C \delta$ since
$|t-s| \le \delta$. For the first term on the right hand side of \eqref{eq:i2}, by \eqref{eq:boundyn}, we have
\begin{equation*}
  \begin{split}
    \int_s^t du \left(V'(r_n(u))- \bar\tau(u) \right) =
    \frac{p_n(s) - p_n(t)}{n^2} - \gamma  \int_s^t p_n(u) du +
    \frac{\sqrt{2\gamma\beta_n^{-1}}}{n} \left(w_n(t) - w_n(s)\right) 
  \end{split}
\end{equation*}
The last term of the right hand side of the above is estimated by the
standard modulus of continuity of the Wiener process $w_n$. For the
second term of the right hand side, this is bounded by 
\begin{equation*}
  \begin{split}
    \mathbb E \left[\sup_{0\le s<t\le T, |s-t| <\delta} \gamma
      \left|\int_s^t p_n(u) du \right|\right] \le \gamma \delta^{1/2}
    \mathbb E \left[\left(\int_0^T p_n^2 (u)
        du\right)^{1/2}\right]\\
     \le \gamma \delta^{1/2} \left[\int_0^T  \mathbb E( p_n^2 (u))
        du\right]^{1/2} = \gamma \delta^{1/2} \left[\int_0^T  \mathbb E( p_n^2 (u) - \beta_n^{-1})  
        du + T \beta_n^{-1} \right]^{1/2}\\
      \le C \gamma \delta^{1/2} \left[\int_0^T \int p_n \partial_{p_n}
        f^n_u d\nu_{\beta_\cdot}  du + T \beta_n^{-1} \right]^{1/2} \\
      \le  C \gamma \delta^{1/2} \left[\left(\int_0^T \int p_n^2 f^n_u
          d\nu_{\beta_\cdot}  du \right)^{1/2} \left(\int_0^T \int
          \frac{(\partial_{p_n} f_u^n)^2}{f_u^n} d\nu_{\beta_\cdot} du \right)^{1/2}
        + T \beta_n^{-1} \right]^{1/2} \\
      \le  C' \gamma \delta^{1/2}
  \end{split}
\end{equation*}
where the last inequality is justified by the inequalities:
\begin{equation*}
  \begin{split}
    \int p_n^2 f^n_u d\nu_\cdot \le Cn\\
    \int_0^T \int
          \frac{(\partial_{p_n} f_u^n)^2}{f_u^n} d\nu_\cdot du \le
          \frac Cn
  \end{split}
\end{equation*}

To deal with the first term we have to prove that 
\begin{equation}\label{eq:122}
  \begin{split}
    \lim_{n\to\infty} \mathbb E\left(\sup_{0\le t\le T}
      \frac 1{n^2}|p_n(t)|\right)  = 0
  \end{split}
\end{equation}
Since
\begin{equation}
  \label{eq:12}
  \begin{split}
    \frac{p_n(t)}{n^2} = \frac 1{n^2} p_n(0) e^{-\gamma n^2t} +
    \int_0^t e^{-\gamma n^2 (t-u)} \left[\bar\tau(u) -
      V'(r_n(u))\right] du \\
    + \sqrt{2\gamma\beta^{-1}_n}\frac 1n \int_0^t e^{-\gamma n^2 (t-u)}
    dw_n(u)
  \end{split}
\end{equation}
The stochastic integral is easy to estimate by Doob's inequality:
\begin{equation*}
  \mathbb E\left(\sup_{0\le t\le T} \left|\sqrt{2\gamma\beta^{-1}_n}\frac 
    1n \int_0^t e^{-\gamma n^2 (t-u)} dw_n(u)\right|^2\right) \le
\frac{CT}{n^2}  
\end{equation*}
About the second term, by Schwarz inequality we have that
\begin{equation*}
  \begin{split}
    \mathbb E\sup_{0\le t\le T} &\left|\int_0^t e^{-\gamma n^2 (t-u)}
      \left[\bar\tau(u) - V'(r_n(u))\right] du\right| \\
    &\le \frac{1}{n\sqrt{2\gamma}} \left( \int_0^T \mathbb
      E\left(\left[\bar\tau(u) - V'(r_n(u))\right]^2\right)
      du\right)^{1/2}
  \end{split}
\end{equation*}
and by the entropy bound we have
\begin{equation*}
  \mathbb E\left(\left[\bar\tau(u) -
      V'(r_n(u))\right]^2\right) \le Cn
\end{equation*}
so that this term goes to zero like $n^{-1/2}$. The first term in
\eqref{eq:12} is trivial to estimate. This conclude the estimate of $I_2$.

The estimation of $I_4$  is similar to the proof of \eqref{eq:122}, but
require a little extra work. We need to prove that 
\begin{equation}\label{eq:forI4}
   \lim_{n\to\infty} \mathbb E\sup_{0\le t\le T} \left| \frac 1{n^2} \sum_{i=2}^{n-1}
     \nabla_n G\left(\frac in\right) p_i(t) \right| = 0.
\end{equation}
 By the evolution equations we have
 \begin{equation*}
   \begin{split}
     \frac 1{n^2} \sum_{i=2}^{n-1}
     \nabla_n G\left(\frac in\right) p_i(t) = \frac 1{n^2} \sum_{i=2}^{n-1}
     \nabla_n G\left(\frac in\right) p_i(0)e^{-\gamma n^2 t} \\
     + \int_0^t ds\ e^{-\gamma n^2 (t-s)}  \frac 1{n} \sum_{i=3}^{n-1}
     \nabla_n^*\nabla_n G\left(\frac in\right) V'(r_i(s)) \\
     + \int_0^t ds\ e^{-\gamma n^2 (t-s)}  
     \left(\nabla_n G\left(1\right) V'(r_n(s)) - \nabla_n G\left(\frac
       2n\right) V'(r_2(s)) \right) 
   \end{split}
 \end{equation*}
and all these terms can be estimated as in the proof of
\eqref{eq:122}, so that \eqref{eq:forI4} follows.

Also $I_5$ can be easily estimated by Doob inequality and using the
independence of $w_i(t)$. 

Finally estimating $I_1$, notice that since $G$ is a smooth function, it can be bounded by
\begin{equation}
  \label{eq:43}
  \begin{split}
    \sup_{0\le s<t\le T, |s-t| <\delta} |I_1(s,t)| &\le \frac C{\gamma
      n} \sup_{0\le s<t\le T, |s-t| <\delta} \int_s^t du
    \sum_{i=2}^{n-1} |V'(r_{i+1}(u))|\\
    &\le  \frac {C\delta^{1/2}}{\gamma} \left(\int_0^T \frac 1n  \sum_{i=2}^{n-1} |V'(r_{i+1}(u))|^2 du\right)^{1/2}
  \end{split}
\end{equation}
and, by entropy inequality
\begin{equation*}
  \mathbb E\left[\left(\int_0^T \frac 1n  \sum_{i=2}^{n-1} |V'(r_{i+1}(u))|^2 du\right)^{1/2}\right] \le
 \left[  \int_0^T \frac 1n  \sum_{i=2}^{n-1} \mathbb E \left( |V'(r_{i+1}(u))|^2\right) du \right]^{1/2} \le C, 
\end{equation*}
so that the expression in \eqref{eq:43} is negligible after $\delta\to 0$. 
\end{proof}
\section{Appendix C: Uniqueness of weak solutions}
\label{sec:append-c:-uniq}

\begin{prop}
  The weak solution of \eqref{eq:weak} is unique in the class of
  function such that
  \begin{equation}
    \label{eq:2}
    \int_0^t ds \int_0^1 \left(\partial_x\tau(r(x,s),
      \beta(x))\right)^2 dx < +\infty 
  \end{equation}
\end{prop}

\begin{proof}
Let $g(x)\ge 0$ a smooth function with compact support contained in
$[-1/4, 1/4]$ such that $\int_{\mathbb R} g(y) dy =1$. Then for
$\lambda>0$ large enough, define the function
\begin{equation*}
  G_\lambda(y,x) = 1 - \int_{-\infty}^{y}\lambda g(\lambda (z-x)) dz
\end{equation*}
Then for $1/(4\lambda) <x < 1- 1/(4\lambda)$, we have $G_\lambda(0,x)
= 0$ and $\partial_y G_\lambda(1,x)=0$, and it can be used as test
function in \eqref{eq:weak}. So if $r(x,t)$ is a solution in the given
class, we have
\begin{equation*}
  \int_0^1 G_\lambda(y,x) \left(r(y,t) - r_0(y)\right) dx 
  = \gamma^{-1}\int_0^t ds  \left[\int_0^1 \lambda g(\lambda(y-x)) 
    \partial_y\mathbf{\tau}(r(y,s), \beta(y)) dy \right].
\end{equation*}
Letting $\lambda \to +\infty$ we obtain:
\begin{equation*}
  \int_0^x  \left(r(y,t) - r_0(y)\right) dx 
  = \gamma^{-1}\int_0^t ds 
    \partial_y\mathbf{\tau}(r(x,s), \beta(x)), \qquad \forall x\in (0,1).
\end{equation*}

Let $r_1(x,t), r_2(x,t)$ two solutions in the class considered, and define
\begin{equation*}
  R_j(x,t) = \int_0^x r_j(y,t) dy, \qquad j=1,2.
\end{equation*}
By the approximation argument done at the beginning of the proof,
 we have that
\begin{equation*}
  \partial_t  R_j(x,t) = \gamma^{-1}\partial_x\tau(r_j(x,s),\beta(x))
\end{equation*}
for every $x\in (0,1)$ and $t>0$.

Since $\tau( r_j(1,t),\beta(1))  = \bar\tau(t)$, and since
$\tau(r,\beta)$ is a strictly increasing function of $r$,
\begin{equation*}
  \begin{split}
   & \frac{d}{dt} \int_0^1 \left( R_1(x,t) - R_2(x,t)\right)^2 dx\\ & = 
    2 \gamma^{-1} \int_0^1 \left( R_1(x,t) -
      R_2(x,t)\right) \partial_x  \left( 
     \tau( r_1(x,t),\beta(x))  - \tau( r_2(x,t),\beta(x))
   \right) dx \\
&= -2 \gamma^{-1} \int_0^1 \left( r_1(x,t) - r_2(x,t)\right) \left(
     \tau( r_1(x,t), \beta(x))  - \tau( r_2(x,t),\beta(x))
   \right) dx \le 0.
  \end{split}
\end{equation*}
\end{proof}

\section{Appendix D: proof of the entropic hypocoercive bound (4.16) }
\label{sec:appendix-d:-proof}

% In order to shorten formulas, we introduce here some vectorial
% notation:
% \begin{equation*}
%  u \bdot v = \sum_{i=1}^{n} \beta^{-1}_{i} u_i v_i,\qquad  u
%  \tilde\bdot v = \sum_{i=1}^{n-1} \beta^{-1}_{i} u_i v_i, \qquad
%  |u|_\bdot^2 =  u \bdot u, \quad 
%   |u|_{\tilde\bdot}^2 =  u \tilde\bdot u.  
% \end{equation*}

% \begin{equation}
%   \label{eq:23}
%   \begin{split}
%     \partial_p = (\partial_{p_1}, \dots, \partial_{p_n}) \qquad 
%      \partial^*_p = (\partial^*_{p_1}, \dots,
%      \partial^*_{p_n}), 
%      \qquad \partial^*_{p_i} = \beta_i p_i - \partial_{p_i}\\
%     \partial_q = (\partial_{q_1}, \dots, \partial_{q_n}),
%     \qquad \partial_{q_i} = \partial_{r_i} - \partial_{r_{i+1}}, 
%     \quad \partial_{q_n} = \partial_{r_n}
%        \end{split}
% \end{equation}

% Using this notation we have:
% \begin{equation}
%   \label{eq:21}
%   \begin{split}
%     I_n(f) &= \int \frac{|\partial_p f+\partial_q f|^2_{\tilde\bdot}}{f} 
%    d\nu_{\beta\cdot} \\
%     &=
%     \sum_{i=1}^{n-1} \left( \cD_{i,n}^p(f) + \cD_{i,n}^r(f) \right)+
%       2\int \frac{\partial_{q} f \tilde\bdot \partial_{p} f}{f} 
%       d\nu_{\beta\cdot} \ge 0
%   \end{split}
% \end{equation}
% Furthermore
% \begin{equation}
%   \label{eq:25}
% \mathcal S_n = - \partial^*_p \odot \partial_p
% - \beta_n^{-1} \partial^*_{p_n} \partial_{p_n}, \qquad  
%   \mathcal A_n^\tau = p\cdot \partial_q - \partial_q \mathcal V
%   \cdot \partial_p  + \tau \partial_{p_n}
% \end{equation}
% where $\mathcal V = \sum_i V(r_i)$ and the $\cdot$ denotes the usual
% scalar product in $\mathbb R^n$. 

We will prove in this appendix that there exists constants $\lambda>0$ and $C>0$
independent of $n$ such that
\begin{equation}
  \label{eq:22}
  \frac{d}{dt} I_n(f) \le - \lambda n^2 I_n(f) + C n.
\end{equation}

We will use the following
commutation relations:
\begin{equation}
  \label{eq:comm}
 [\partial_{p_i}, \beta_j^{-1}\partial_{p_j}^*] = \delta_{i,j}, \qquad 
 [\partial_{p_i}, \mathcal A_n^\tau] = \partial_{q_i}, \qquad [\partial_{q_i},
  \mathcal A_n^\tau] = - (\partial_q^2 \mathcal V\; \partial_p)_i 
\end{equation}
where $\partial_q^2 \mathcal V$ is the corresponding hessian matrix of $\mathcal V = \sum_{i=1}^n V(r_n)$.

Denote $g_t = \sqrt{f^n_t}$ and observe that
\begin{equation}\label{eq:Ig}
  I_n(g_t^2) = 4 \int \left( |\partial_p g_t|_{\tdot}^2 + 
    |\partial_q g_t|_\tdot^2 
    + 2 \; \partial_q g_t \tdot \partial_p g_t \right) d\nu_{\beta\cdot}
\end{equation}
Recall that 

\begin{equation}
n^2\cA_n^{\tau,*} = -n^2\cA_n^{\tau} + B_n^{\tau}\label{eq:3}
\end{equation}
where 
$$
B_n^\tau = - n\sum_{i=1}^{n-1} \nabla_n \beta(i/n) p_iV'(r_{i+1}) +
      n^2\beta(1)p_n {\tau}
$$
Consequently $g_t$ solves the equation:
\begin{equation*}
  \begin{split}
    \partial_t g = - n^2 \cA_n^{\bar\tau(t)}g_t + n^2\gamma \cS_n g_t +
    n^2\gamma\frac{|\partial_p g_t|_\odot^2}{g_t} % + 
    % n^2\gamma\beta_n^{-1} \frac{(\partial_{p_n} g_t)^2}{g_t}
    + \frac 12 B_n^{\bar\tau(t)}  g_t
  \end{split}
\end{equation*}

We then compute the time derivative of $I_n(g_t^2)$ by considering the
three terms separately. The first one gives:

\begin{equation}\label{eq:Ip}
  \begin{split}
    \frac d{dt} \int |\partial_p g_t|_{\tdot}^2 \ d\nu_{\beta\cdot} =& -
    2 n^2\int \partial_p g_t \tdot \partial_p (\mathcal A^{\bar\tau(t)}
    g_t)  \ d\nu_{\beta\cdot}\\
    & - 2 n^2\gamma\int \partial_p g_t \tdot \partial_p (\partial_p^*
    \odot \partial_p 
     g_t)  \ d\nu_{\beta\cdot}\\
     &+ 2 n^2\gamma \int \partial_p g_t \tdot \partial_p
     \left(\frac{|\partial_p 
       g_t|^2_\odot}{ g_t}\right)  \ d\nu_{\beta\cdot}\\
   &+ \int \partial_p g_t \tdot \partial_p (B_n^{\bar\tau(t)}
    g_t)  \ d\nu_{\beta\cdot} .
  \end{split}
\end{equation}
By the commutation relations \eqref{eq:comm}, and using \eqref{eq:3},
 the first term on the RHS of \eqref{eq:Ip} is equal to
\begin{equation*}
  \begin{split}
    - 2 n^2 \int \partial_p g_t \tdot \partial_q g_t \
    d\nu_{\beta\cdot}  - 2 n^2\int \partial_p g_t \tdot \mathcal A^{\bar\tau(t)}
    \partial_p g_t  \ d\nu_{\beta\cdot}\\
    = - 2 n^2 \int \partial_p g_t \tdot \partial_q g_t \
    d\nu_{\beta\cdot}  - \int \partial_p g_t \tdot B_n^{\bar\tau(t)}
    \partial_p g_t  \ d\nu_{\beta\cdot}
  \end{split}
\end{equation*}

Then the RHS of \eqref{eq:Ip} is equal to
\begin{equation*}
  \begin{split}
   &- 2 n^2 \int \partial_p g_t \tdot \partial_q g_t \
    d\nu_{\beta\cdot}  - 2 n^2\gamma\int \partial_p g_t
    \tdot \partial_p (\partial_p^* 
    \odot \partial_p 
     g_t)  \ d\nu_{\beta\cdot}\\
     &+ 2 n^2\gamma \int \partial_p g_t \tdot \partial_p
     \left(\frac{|\partial_p 
       g_t|^2_\odot}{ g_t}\right)  \ d\nu_{\beta\cdot}
   + \int g_t \partial_p g_t \tdot \partial_p B_n^{\bar\tau(t)} \
   d\nu_{\beta\cdot} . 
  \end{split}
\end{equation*}
The last term of the above equation is equal to
\begin{equation}
  \label{eq:abt}
  \begin{split}
    \int g_t \partial_p g_t \tdot \partial_p B_n \ d\nu_{\beta\cdot} =
    - n\int g_t \sum_{i=1}^{n-1} \beta_i^{-1} \nabla_n\beta(\frac in)
    V'(r_{i+1}) \partial_{p_i} g_t \ d\nu_{\beta\cdot} % \\
    % + \frac{n^2}2 \beta_n \bar\tau(t) \int p_n g_t^2 \
    % d\nu_{\beta\cdot} 
  \end{split}
\end{equation}
Notice that the term involving $n^2 \tau p_n$ does not appear in the above expression, because the particular definition of $\tdot$.
For any $\alpha_1 >0$,  using Schwarz inequality, \eqref{eq:energyb1} and \eqref{eq:energyb2}, \eqref{eq:abt} is bounded by
\begin{equation*}
  \begin{split}
    \frac 1{2\alpha_1} \int g_t^2 \sum_{i=1}^{n-1}
    \frac{(\nabla_n\beta(\frac in))^2}{\beta_i} V'(r_{i+1})^2  \
    d\nu_{\beta\cdot}  + \frac{\alpha_1 n^2}2
  \int |\partial_p g_t|_{\tdot}^2 \ d\nu_{\beta\cdot}  \\
  \le \frac{Cn}{\alpha_1} +
  \frac{\alpha_1 n^2}2 
  \int |\partial_p g_t|_{\tdot}^2 \ d\nu_{\beta\cdot} 
  \end{split}
\end{equation*}
for a constant $C$ depending on $\beta_{\cdot}$ and the initial entropy, but independent of $n$. 

Computing the second term of the RHS of \eqref{eq:Ip} we have:
\begin{equation*}
  \begin{split}
    \int \partial_p g_t \tdot \partial_p (\partial_p^*
    \odot \partial_p g_t) \ d\nu_{\beta\cdot} = \int \sum_{j=1}^{n-1}
    \beta_j^{-1} |\partial_p \partial_{p_j} g|_{\odot}^2 \
    d\nu_{\beta\cdot} + \int |\partial_p g|_{\tdot}^2 \
    d\nu_{\beta\cdot} \\
    = \int \sum_{i=1}^{n}\sum_{j=1}^{n-1}
    \beta_j^{-1} \beta_i^{-1} (\partial_{p_i} \partial_{p_j} g)^2 \
    d\nu_{\beta\cdot} + \int |\partial_p g|_{\tdot}^2 \
    d\nu_{\beta\cdot} 
  \end{split}
\end{equation*}
About the third term on the RHS:
\begin{equation*}
  \begin{split}
    \partial_p g_t \tdot \partial_p \left(\frac{|\partial_p
        g_t|^2_\odot}{ g_t}\right) = \frac{2 \sum_{j=1}^{n-1} \sum_{i=1}^n
      \beta_j^{-1}\beta_i^{-1}  
      \;\partial_{p_j} g_t
      \; \partial_{p_i}g_t\; \partial_{p_i} \partial_{p_j} g_t}{g_t} -  
    \frac{|\partial_pg|_{\odot}^2 |\partial_pg|_{\tdot}^2}{g_t^2} 
  \end{split}
\end{equation*}
Summing all together we have obtained
\begin{equation}
  \label{eq:26}
  \begin{split}
     \frac d{dt} \int |\partial_p g_t|_{\tdot}^2 \ d\nu_{\beta\cdot}
     = - 2 n^2 \int \partial_p g_t \tdot \partial_q g_t \
    d\nu_{\beta\cdot} -  n^2\left(2\gamma - \frac {\alpha_1}2\right)
    \int |\partial_p g_t|_{\tdot}^2 \ 
    d\nu_{\beta\cdot}  \\
     - 2 n^2\gamma\int \sum_{j=1}^{n-1} \sum_{i=1}^n
    \beta_j^{-1}\beta_j^{-1} 
    \left(\partial_{p_i} \partial_{p_j} g_t -
      g_t^{-1}\partial_{p_i}g_t \partial_{p_j} 
    g_t\right)^2 \ d\nu_{\beta\cdot} + \frac{Cn}{\alpha_1} .
  \end{split}
\end{equation}

Now we deal with the derivative of the second term:
\begin{equation}
  \label{eq:24}
  \begin{split}
    \frac d{dt} \int |\partial_q g_t|_{\tdot}^2 \ d\nu_{\beta\cdot} =
    - 2 n^2\int \partial_q g_t \tdot \partial_q (\mathcal
    A^{\bar\tau(t)}
    g_t)  \ d\nu_{\beta\cdot}
     - 2 n^2\gamma\int \partial_q g_t \tdot \partial_q (\partial_p^*
    \odot \partial_p g_t)  \ d\nu_{\beta\cdot}\\
    + 2 n^2\gamma \int \partial_q g_t \tdot \partial_q
    \left(\frac{|\partial_p
        g_t|^2_\odot}{ g_t}\right)  \ d\nu_{\beta\cdot} 
    + \int \partial_q g_t \tdot \partial_q (B_n g_t) \
    d\nu_{\beta\cdot} \\
    =
    - 2 n^2\int \partial_q g_t \tdot \partial_q (\mathcal
    A^{\bar\tau(t)}
    g_t)  \ d\nu_{\beta\cdot} - 2 n^2\gamma\int \sum_{j=1}^{n-1}
    \sum_{i=1}^n 
    \beta_i^{-1}\beta_j^{-1} 
    \left(\partial_{p_i} \partial_{q_j} g -
      g_t^{-1}\partial_{p_i}g \partial_{q_j} 
    g\right)^2 \ d\nu_{\beta\cdot}\\
    + \int \partial_q g_t \tdot \partial_q (B_n g_t) \
    d\nu_{\beta\cdot} .
  \end{split}
\end{equation}
The first and the last term give:
\begin{equation*}
  \begin{split}
    - 2 n^2\int \partial_q g_t \tdot \partial_q (\mathcal
    A^{\bar\tau(t)} g_t) \ d\nu_{\beta\cdot} + \int \partial_q g_t
    \tdot \partial_q (B_n g_t) \ d\nu_{\beta\cdot}\\
    = 2 n^2\int \partial_q g_t \tdot (\partial_q^2 \mathcal
    V \partial_p) g_t \ d\nu_{\beta\cdot} + \int g_t \partial_q g_t
    \tdot \partial_q B_n \ d\nu_{\beta\cdot}
  \end{split}
\end{equation*}
The last term on the RHS of the above expression is equal to 
\begin{equation*}
  \begin{split}
    &\int g_t \partial_q g_t \tdot \partial_q B_n \ d\nu_{\beta\cdot}
    \\ = &
    n \sum_{i=2}^{n-1} \int \beta_i^{-1} g_t (\partial_{q_i}g_t)
    \left[\nabla_n \beta\left(\frac in\right) V''(r_{i+1}) p_i - \nabla_n
      \beta\left(\frac{i-1}n\right) V''(r_{i}) p_{i-1} \right] \
    d\nu_{\beta\cdot}\\
   &  + n  \int \beta_1^{-1} g_t (\partial_{q_1}g_t) \nabla_n \beta\left(\frac 1n\right) V''(r_{2}) p_1 .
  \end{split}
\end{equation*}
Since $V''$ and $\nabla_n \beta$ are bounded and $\beta(\cdot)$ is positive bounded away from $0$,
this last quantity is bounded for any $\alpha_2>0$ by  
\begin{equation*}
  n^2 \alpha_2 \int |\partial_q g_t|_\tdot^2  d\nu_{\beta\cdot} + C\alpha_2^{-1}
  \int \sum_{i=1}^{n-1} p_i^2 g^2_t  d\nu_{\beta\cdot} \le n^2 \alpha_2 \int
  |\partial_q g_t|_\tdot^2  d\nu_{\beta\cdot} + C'\alpha_2^{-1} n.
\end{equation*}
Since $V''$ is bounded, for any $\alpha_3 >0$ we have
\begin{equation*}
  2 n^2\int \partial_q g_t \tdot (\partial_q^2 \mathcal
    V \partial_p) g_t \ d\nu_{\beta\cdot} \le \alpha_3
    n^2\int |\partial_q g_t|^2_{\tdot}  \ d\nu_{\beta\cdot} + 
    \frac{|V''|_\infty^2 n^2}{\alpha_3} \int |\partial_p g_t|^2_{\tdot} \ d\nu_{\beta\cdot}
\end{equation*}
Putting all the terms together, the time derivative of the second term is bounded by
\begin{equation}
  \label{eq:27}
  \begin{split}
      \frac d{dt} & \int |\partial_q g_t|_{\tdot}^2 \ d\nu_{\beta\cdot}
      \le  (\alpha_2 + \alpha_3)
    n^2\int |\partial_q g_t|^2_{\tdot}  \ d\nu_{\beta\cdot} + 
    \frac{Cn^2}{\alpha_3} \int |\partial_p g_t|^2_{\tdot} \
    d\nu_{\beta\cdot}\\
    &  - 2 n^2\gamma\int \sum_{j=1}^{n-1} \sum_{i=1}^n
    \beta_i^{-1}\beta_j^{-1} 
    \left(\partial_{p_i} \partial_{q_j} g -
      g_t^{-1}\partial_{p_i}g \partial_{q_j} 
    g\right)^2 \ d\nu_{\beta\cdot} + C'\alpha_2^{-1} n
  \end{split}
\end{equation}

About the derivative of the third term, using the third of the commutation relations \eqref{eq:comm}, gives
\begin{equation}
  \label{eq:28}
  \begin{split}
    \frac d{dt} 2\int \partial_q g_t \tdot \partial_p g_t
    d\nu_{\beta\cdot} = -2n^2 \int \left[\partial_q (\mathcal A^{\tau(t)}g_t)
    \tdot \partial_p g_t + 
    \partial_q  g_t 
    \tdot \partial_p (\mathcal A^{\tau(t)}g_t)\right] d\nu_{\beta\cdot}\\
    +  \int \left[\partial_q (B_n g_t) \tdot \partial_p g_t
    + \partial_q g_t \tdot \partial_p (B_n g_t)
  \right]d\nu_{\beta\cdot}\\
  - 2n^2\gamma\int \left[\partial_q g_t \tdot \partial_p (\partial_p^*
    \odot \partial_p g_t) + \partial_q (\partial_p^*
    \odot \partial_p g_t) \tdot \partial_p g_t \right]\
  d\nu_{\beta\cdot}\\
  +2 n^2\gamma \int \left[\partial_q g_t \tdot \partial_p
    \left(\frac{|\partial_p g_t|^2_\odot}{ g_t}\right)  
    + \partial_q 
    \left(\frac{|\partial_p g_t|^2_\odot}{g_t}\right) \tdot \partial_p
    g_t\right] \ d\nu_{\beta\cdot} \\
 = 2n^2 \int (\partial_q^2\mathcal V \partial_p) g_t
    \tdot \partial_p g_t d\nu_{\beta\cdot} 
    - 2 n^2 \int |\partial_q g_t|^2_\tdot d\nu_{\beta\cdot} \\
    + \frac 12 \int g_t\left[\partial_q B_n \tdot \partial_p g_t
    + \partial_q g_t \tdot \partial_p B_n \right]d\nu_{\beta\cdot}\\
  - 4n^2\gamma\int\sum_{j=1}^{n-1} \sum_{i=1}^n \beta_i^{-1}\beta_j^{-1} 
    \left[\left(\partial_{p_i} \partial_{q_j} g\right)
      \left(\partial_{p_i} \partial_{p_j} g\right)\right]
    d\nu_{\beta\cdot}\\ 
    +2 n^2\gamma\int \sum_{j=1}^{n-1} \sum_{i=1}^n
    2\beta_i^{-1}\beta_j^{-1}  g_t^{-1}
    \left[(\partial_{p_i} \partial_{p_j} g) (\partial_{p_i} g_t) 
      (\partial_{q_j} g_t) + (\partial_{q_j} \partial_{p_i} g)
      (\partial_{p_j} g_t) (\partial_{p_i} g_t) \right]
    d\nu_{\beta\cdot}\\
    - 4n^2 \gamma\int \sum_{j=1}^{n-1} \sum_{i=1}^n
    \beta_i^{-1}\beta_j^{-1} g^{-2} 
    (\partial_{p_i} g_t)^2 (\partial_{q_j} g_t) (\partial_{p_j} g_t) 
    d\nu_{\beta\cdot}
  \end{split}
\end{equation}

The last three terms of the RHS of the \eqref{eq:28} can be written as
\begin{equation*}
  \begin{split}
    - 4n^2\gamma\int \sum_{j=1}^{n-1} \sum_{i=1}^n
    \beta_i^{-1}\beta_j^{-1}  
    \left[\left(\partial_{p_i} \partial_{q_j} g_t - g_t^{-1}\partial_{p_i} g_t 
      \partial_{q_j} g_t\right)
    \left(\partial_{p_i} \partial_{p_j} g_t - g_t^{-1} \partial_{p_i} g_t  
      \partial_{p_j} g_t\right) \right]
    d\nu_{\beta\cdot}
  \end{split}
\end{equation*}
so they combine with the corresponding terms coming from the time
derivative of the first two terms of $I_n$ giving an exact square.

The second term of \eqref{eq:28}, by the same arguments used before, can be bounded by
\begin{equation*}
  \begin{split}
    n^2 \alpha_4 \int |\partial_q g_t|^2_\tdot d\nu_{\beta\cdot} + 
    n^2 \alpha_5 \int |\partial_p g_t|^2_\tdot d\nu_{\beta\cdot}
    +  Cn (\alpha_4^{-1} + \alpha_5^{-1})
  \end{split}
\end{equation*}

 About the first term of \eqref{eq:28}, since $V''$ is bounded, it is
bounded by $V''_\infty n^2 \int |\partial_p g_t|^2_\tdot
d\nu_{\beta\cdot}$.

Putting all these bounds together we obtain that
\begin{equation*}
  \begin{split}
     \frac{d}{dt} I_n(f_t) \le - n^2 \kappa_p  \int |\partial_p
     g_t|^2_\tdot  d\nu_{\beta\cdot} - n^2 \kappa_q  \int |\partial_q
     g_t|^2_\tdot  d\nu_{\beta\cdot} - 2n^2 \int \partial_p g_t
     \tdot \partial_q g_t  d\nu_{\beta\cdot}  + Cn \\
     - 2 N^2 \gamma \int \sum_{j=1}^{n-1} \sum_{i=1}^n \beta_i^{-1}\beta_j^{-1}  
    \left[\left(\partial_{p_i} \partial_{q_j} g_t - g_t^{-1}\partial_{p_i} g_t 
      \partial_{q_j} g_t\right) +
    \left(\partial_{p_i} \partial_{p_j} g_t - g_t^{-1} \partial_{p_i} g_t  
      \partial_{p_j} g_t\right) \right]^2 
    d\nu_{\beta\cdot}
  \end{split}
\end{equation*}
with
\begin{equation*}
  \begin{split}
    \kappa_p &= 2\gamma - \frac{\alpha_1}2 - \frac C{\alpha_3} -
    \alpha_5 - V''_{\infty} \\
    \kappa_q &= 2 - \alpha_2 - \alpha_3 - \alpha_4
  \end{split}
\end{equation*}
By choosing $\alpha_2 + \alpha_3 +\alpha_4 \le 1$
we have obtained that for some constants $C_1, C_2 >0$ independent of $n$
\begin{equation*}
   \frac{d}{dt} I_n(f_t) \le - n^2 I_n(f_t) + C_1 n + C_2 \int
   |\partial_p g_t|^2_\tdot  d\nu_{\beta\cdot} .
\end{equation*}
By recalling that 
\begin{equation*}
  \int_0^t  ds \int |\partial_p g_s|^2_\tdot  d\nu_{\beta\cdot} \le \frac{C'}{n}
\end{equation*}
 after time integration we have for some constant $C_3$:
\begin{equation*}
   I_n(f_t) \le e^{- n^2 t} I_n(f_0) + \frac {C_3}{n} (1- e^{-n^2t}) 
\end{equation*}
that implies
\begin{equation}
  \label{eq:29}
  I_n(f_t) \le \frac {C_4}n
\end{equation}
for any reasonable initial conditions such that $I_n(f_0)$ is finite
and not growing too fast with $n$.

% Since we already know that $\int |\partial_p g_t|^2_\odot
% d\nu_{\beta\cdot} \le Cn^{-1}$, it follows easily  that
% \begin{equation}
%   \label{eq:30}
%   \int |\partial_q g_t|^2_\odot d\nu_{\beta\cdot} \le Cn^{-1}
% \end{equation}

\begin{rem}
  An important example for understanding the meaning of a density with
  small $I_n$ functional, consider the inhomogeneous Gibbs density:
  \begin{equation}
    \label{eq:8}
    f = \exp\left(\sum_{i=1}^n \beta_i \tau_i r_i + \sum_{i=1}^{n-1} \frac 1n
      \nabla_n(\beta_i\tau_i) p_i\right) / \mathcal N
  \end{equation}
where $\mathcal N$ is a normalization constant.
In the case of constant temperature these densities play an
important role in the relative entropy method (cf
\cite{Tremoulet:2002p20753, olla2014micro2}), as to a non-constant
profile of tension corresponds a profile of small damped velocities
averages. Computing $I_n$ on $f$ we have
\begin{equation*}
  I_n(f) = \sum_{i=1}^{n-1} \left[\beta_i \tau_i - \beta_{i+1} \tau_{i+1} +
  \frac 1n  \nabla_n(\beta_i\tau_i)\right] = 0 .
\end{equation*}

\end{rem}

\nocite{*}
\bibliographystyle{plain}	
\bibliography{iso_grad}

\end{document}